\documentclass[12pt, a4paper, reqno, twoside, makeidx]{amsart}
\usepackage[margin=2.8cm, marginpar=3cm]{geometry}

\usepackage{verbatim}
\usepackage{amsmath}
\usepackage{amsfonts}
\usepackage{amssymb}
\usepackage{graphicx}
\usepackage{color}
\usepackage{empheq}
\usepackage[font={small,it}]{caption}
\usepackage{xypic} 
\usepackage[all]{xy}
\usepackage{tikz-cd}
\usepackage[autostyle]{csquotes}
\usepackage[makeroom]{cancel}

\usepackage{tikz}
\usetikzlibrary{calc}

\def\spinc{\text{spin}^c}

\usepackage{amssymb}

\usepackage{tikz}
\usetikzlibrary{calc}

\theoremstyle{remark}

\theoremstyle{plain}
\newtheorem{thm}{Theorem}[section]
\newtheorem{lem}[thm]{Lemma}
\newtheorem{prop}[thm]{Proposition}
\newtheorem{defi}[thm]{Definition}
\newtheorem{cor}[thm]{Corollary}

\newtheorem{exa}[thm]{Example}
\newtheorem{rem}{Remark}[equation]

\theoremstyle{definition}
\newtheorem{rmk}[thm]{Remark}

\newcommand{\Z}{\mathbb{Z}}
\newcommand{\Q}{\mathbb{Q}}
\newcommand{\R}{\mathbb{R}}
\newcommand{\C}{\mathbb{C}}
\newcommand{\Kt}{\text{tHFK}^-}

\newcommand{\Tt}{\mathbb{T}}
\newcommand{\F}{\mathbb{F}}

\newcommand{\CP}{\mathbb{C} P}
\newcommand{\CF}{\mathbb{C} \mathbb{F}}

\newcommand{\x}{\mathbf{x}}
\newcommand{\y}{\mathbf{y}}
\newcommand{\gr}{\text{gr}}
\newcommand{\boa}{\boldsymbol{\alpha}}
\newcommand{\bob}{\boldsymbol{\beta}}
\newcommand{\bog}{\boldsymbol{\gamma}}
\newcommand{\boe}{\boldsymbol{\eta}}
\newcommand{\bod}{\boldsymbol{\delta}}
\newcommand{\s}{\mathfrak{s}}

\newcommand{\Sym}{\text{Sym}}
\newcommand{\T}{\mathbb{T}}

\newcommand{\Spinc}{\text{Spin}^c}
\newcommand{\dimm}{\text{dim}}

\newcommand{\Char}{\text{Char}}

\newcommand{\W}{q}

\begin{document}
\title[]{Deformations of lattice cohomology \\and the upsilon invariant}
\author{Antonio Alfieri}
\begin{abstract}We introduce deformations of lattice cohomology corresponding to the knot homologies found by Ozsv\' ath, Stipsicz and Szab\' o in \cite{OSS4}. By means of holomorphic triangles counting, we prove equivalence with the analytic theory for a wide class of knots. This yields combinatorial formulae for the upsilon invariant.
\end{abstract}
\maketitle
\thispagestyle{empty}
\section{Introduction}
Based on the methods used by Floer \cite{floer} in Symplectic Topology to study the intersection properties of  Lagrangian submanifolds, in \cite{OS2,OS4} Ozsv\' ath and Szab\' o introduced a package of three-manifold invariants called Heegaard Floer homology. This eventually led to the definition of Knot Floer homology \cite{OS7,Ras1}, a related package of knot invariants.

In the last two decades Knot Floer homology proved to be an extremely powerful tool for the study of knots in $S^3$ \cite{tau,surgery,YN}. In particular, in the realm of knot concordance, invariants like the upsilon invariant $\Upsilon_K(t)$ introduced by Ozsv\'ath, Stipsicz, and  Szab\'o \cite{OSS4} turned out to be extremely efficient to decide certain questions about independence in the knot concordance group $\mathcal{C}$.

In \cite{upsiloncovers}, extending the definition of the upsilon invariant to knots inside rational homology spheres, new invariants for knots in $S^3$ were introduced. Indeed, given a knot $K\subset S^3$ one can consider its branched double cover $\Sigma(K)$. This is a rational homology sphere carrying a unique spin structure $\s_0$. By considering the pull-back $\widetilde{K} \subset \Sigma(K)$ of $K$ to $\Sigma(K)$ we get a null-homologous knot $\widetilde{K} \subset \Sigma(K)$. One can see that the upsilon invariant $\Upsilon_{K, \s_0}(t)$ of $(\Sigma(K), \widetilde{K})$ with respect to $\s_0$ yields a knot concordance invariant of $K \subset S^3$. 

Further invariants carrying information about the concordance type of $K$ can be obtained by considering the invariants $\Upsilon_{K, \s}(t)$ of $\widetilde{K}$ associated to the other $\Spinc$ structures of the double branched cover $\Sigma(K)$. More specifically in \cite{upsiloncovers} the following theorem, reminiscent of the results of Grigsby, Ruberman, and Strle \cite{grigsby2008knot},  was proved.      

\begin{thm}[Alfieri, Celoria \& Stipsicz]
If $K$ is a slice knot then there exists a subgroup $G < H^2(\Sigma(K); \Z)$ of cardinality
$\sqrt{|H^2(\Sigma(K), \Z)|}$ such that $\Upsilon_{K,\s_0+\xi}(t)=0$ for all $\xi\in G$.  
\end{thm}

In \cite{upsiloncovers} we performed  computations of the invariants $\Upsilon_{K,\s_0+\xi}(t)$, $\xi \in H^2(\Sigma^m(K), \Z)$ for some families of knots having genus one doubly-pointed Heegaard diagrams. These are known as $(1,1)$ knots, and were first studied by Rasmussen \cite{ras11}.      

In this paper we address the issue of computations in the case of graph knots. In what follows a \textit{graph knot} is  a knot that can be described by means of a pluming tree $\Gamma$ with one unframed vertex $v_0$, see Figure \ref{trefoil} below. Note that the knot of an irreducible plane curve singularity $f(x,y)=0$ is a graph knot. Indeed, so is its lift to the branched double cover $\Sigma(K)$. (See Proposition \ref{procedure} below.) The study of graph knots was initiated by Ozsv\' ath, Stipsicz, and Szab\' o in \cite{Lspaceknots}.

\begin{thm}\label{maingoal} Let $K$ be a null-homologous graph knot associated to a plumbing $\Gamma$ with unframed vertex $v_0$. Suppose  that $G=\Gamma-v_0$ is a negative-definite plumbing tree with at most two bad vertices. Let $\s$ be a $\Spinc$ structure of $Y(G)$ and $k$ a characteristic vector of the intersection form of the associated plumbing of spheres $X(G)$ representing the $\spinc$ structure $\s$ on the boundary. Then  
\[\Upsilon_{K, \s}(t)=  -2 \min_{x\in \Z^s} \chi_t(x)+\left( \frac{k^2+ |G|}{4} - t \ \frac{k\cdot F-F^2}{2} \right) \ , \]
where $\chi_t$ denotes the twisted Riemann-Roch function  
\[\chi_t(x)=-\frac{1}{2}\left( (k+tv_0^*) \cdot x + x^2 \right)\ ,\]
and $F\in H_2(X(G), \Q)$ a homology class representing $-v_0^*\in H^2(X(G), \Z)$.
\end{thm} 

This leads to related formulae for the $\tau$-invariants introduced by Grigsby, Ruberman, and Strle \cite{grigsby2008knot}. These are related to the $\Upsilon$-invariant via the identity $\tau_{\s}(K)= -\lim_{t\to 0^+}\Upsilon_{K, \s}(t)/t$.

The proof of Theorem \ref{maingoal} we outline presently is an adaptation of the argument presented in \cite{OSGraphManifolds}. (A similar type of work was carried on in \cite{instantons} where the same technique was employed to perform computations in the setting of Instanton Floer homology.) There are two main ingredients. The first is a surgery exact triangle involving the ''$t$-modified'' knot homologies introduced by Ozsv\'ath, Stipsicz, and  Szab\'o in \cite{OSS4}. Compare with the knot Floer exact triangle of Ozsv\' ath and Szab\' o \cite[Theorem 8.2]{OS7}. 

\begin{thm}\label{exactanalyticthm} 
Let $K \subset Y$ be a knot in a rational homology sphere, and $C \subset Y$  a framed knot in its complement. Let $\lambda$ denotes the framing of $C$, and $\mu$ its meridian. Suppose that  the surgery three-manifolds $Y_\lambda(C)$ and $Y_{\lambda+\mu}(C)$ are rational homology spheres. Then there is an exact triangle
\begin{equation}\label{exactanalytic}
\begin{tikzcd}[column sep=small]
 \Kt(Y,K) \arrow{rr} &     & \Kt(Y_\lambda(C), K) \arrow{ld} \\
& \Kt(Y_{\lambda+\mu}(C), K)  \arrow{lu}   & 
\end{tikzcd} \ .
\end{equation}
\end{thm}

Secondly, in the spirit of \cite{Nemethi1,OSS1}, we introduce a combinatorial invariant associated to algebraic knots. This is a one-parameter family of knot homologies $t\mathbb{HFK}_*(\Gamma)$ with the same formal structure as $t\text{HFK}(Y,K)$. In Section \ref{combinatorialexactsequence} we establish a long exact sequence playing the role of the surgery exact triangle in the combinatorial theory.

\begin{thm}\label{latticeexact}There is a long exact sequence of modules
\[ \xymatrix{
 \ar[r] & t\mathbb{HFK}_p(\Gamma-v) \ar[r]^{\ \ \  \phi_*  } &  t\mathbb{HFK}_p(\Gamma) \ar[r]^{\psi_* \ \ \ }  & t\mathbb{HFK}_p(\Gamma_{+1}(v))  \ar[r]^{\ \  \delta} & t\mathbb{HFK}_{p-1}(\Gamma) \ar[r] &} \]
\end{thm}

Our main result is then obtained by comparing the two exact triangles as in \cite{OSGraphManifolds}. Note that the combinatorial theory we develop here builds on the work carried on in \cite{Lspaceknots} by Ozsv\' ath, Stipsicz, and Szab\' o.  

  

\vspace{0.3cm}
\begin{footnotesize}
\textbf{Acknowledgements} 
\textit{I would like to thank Peter Ozsv\' ath for showing interest in this projects while in its early stages, and for some useful advice. I would also like to thank  Andr\' as Stipsicz, Andr\' as Juhasz, Ian Zemke, Daniele Celoria, Andr\' as N\' emethi and Liam Watson for useful conversations. Most of this work was carried on in the summer of 2018 when I  was partially supported by the NKFIH Grant \' Elvonal (Frontier) KKP 126683 and K112735.}
\end{footnotesize}

\section{Embedded resolutions of curves and branched double coverings}
In what follows we will deal with knots in rational homology spheres. As a consequence we will adopt the following terminology.

\begin{defi} A knot is a pair $(Y, K)$ where $Y$ is a smooth closed three-manifold and $K$ is the image of a $C^\infty$ embedding $S^1 \hookrightarrow Y$.
\end{defi} 

A knot $(Y, K)$ is called \textit{null-homologous} if $[K]=0$ in $H_1(Y; \Z)$. This is the same as asking that $K$ has a Seifert surface, \text{i.e.} that there exists a surface with boundary $\Sigma\subset Y$ such that $\partial \Sigma=K$.  
Most of the time in what follows we will deal with null-homologous knots $(Y,K)$ where the three manifolds $Y$ is a rational homology sphere, that is $H_*(Y, \Q)\simeq H_*(S^3,\Q)$.  Furthermore, knots and three-manifolds are always assumed to be \textit{oriented}. Note that in a rational homology sphere knots are guaranteed to be \textit{rationally null-homologous}, that is $[K]=0$ in $H_1(Y; \Q)$.      

Two knots $(Y_0,K_0)$ and $(Y_1,K_1)$ are called \textit{rationally homology concordant} if there is a rational homology cobordism $X:Y_0 \to Y_1$ containing a smoothly embedded cylinder $C \subset X$ such that $\partial C= C \cap \partial X = K_0 \cup-K_1$. If $Y_0$ and $Y_1$ are equipped with $\spinc$ structures, say $\s_0$ and $\s_1$, and these extends over $X$ we say that $(Y_0,K_0, \s_0)$ and $(Y_1,K_1, \s_1)$ are $\Spinc$ rationally homology concordant.

There are basically two ways one can use to represent a knot $(Y,K)$:
\begin{enumerate}
\item via a \textit{mixed diagram} that is a pair $(L,K)$ where $L$ represents a framed link $L$ in the three-sphere $S^3$, and  $K$ a knot lying in the link complement $S^3-L$, 
\item or via a \textit{doubly-pointed Heegaard diagram} that is a Heegaard diagram 
\[(\Sigma, \{\alpha_1, \dots, \alpha_g\},\{\beta_1, \dots, \beta_g\})\] 
together with a pair of base points $z, w\in \Sigma-\alpha_1- \dots- \alpha_g-\beta_1- \dots- \beta_g$ lying in the complement of the $\alpha$- and the $\beta$-curves \cite{OS7}.
\end{enumerate} 

\subsection{Algebraic knots} Let $(C, 0) \subset (\C^2, 0)$ be the germ of an irreducible plane curve singularity. By looking at the intersection of $C \subset \C^2$ with a small sphere $S_\epsilon(0)\subset \C^2$ centred at the origin of the axes we get a knot $(S^3, K)$. Knots of this kind are usually called \textit{algebraic knots}. 

The topology of an algebraic knot $(S^3, K)=(S_\epsilon(0), S_\epsilon(0)\cap C)$ can be understood by means of an embedded resolution of the curve singularity $(C,0)\subset (\C^2, 0)$. By this we mean a complex map $\rho: \C^2 \# n\overline{\CP^2} \to \C^2$ such that:
\begin{itemize}
\item $\rho$ defines an isomorphism $\C^2 \# n\overline{\CP^2} \setminus \rho^{-1}(0) \to \C^2\setminus \{0\} $ away from the origin,  
\item  the \textit{exceptional divisor},
\[E:=\rho^{-1}(C)=E_1 \cup \dots \cup E_n \cup \widetilde{C} \ ,\]
is an algebraic curve with smooth components $\{E_1 , \dots , E_n,\widetilde{C} \}$. Furthermore, $E_i \simeq \CP^1$, $i\in \{1, \dots, n\}$, and $\widetilde{C}=\overline{\rho^{-1}(C\setminus\{0\})}$   
\item no three components of the exceptional divisor pass through the same point, 
\item the exceptional divisor $E$ has only normal crossing singularities, that is: the intersection of two components  of $E$ is locally modelled on \[\{(x,y)\in \C^2 \ | \ x^{m}y^{l}=0\}\] for some $m,l\geq 1$ (the \textit{multiplicities}). 
\end{itemize} 
The exceptional divisor $E$ of an embedded resolution $\C^2 \# n\overline{\CP^2} \to \C^2$ is encoded in the so called \textit{resolution graph}. This is the graph $\Gamma$ having as vertices the irreducible components of $E$ and an edge connecting  each pair of intersecting components. Note that the resolution graph comes with a distinguished vertex (the one corresponding to the proper transform $\widetilde{C}$ of the curve $C$) and integers $e_1,\dots, e_n$ labelling the other vertices. These are defined as the self intersections $e_i=E_i \cdot E_i$ of the curves $E_1 , \dots , E_n$ in the blow-up $\C^2 \# n\overline{\CP^2}$. 

\begin{exa}The curve $C=\{(x,y) \in \C^2 \ : \ x^2+y^3=0\}$ has an isolated singularity at the origin $0\in \C^2$. After  three consecutive blow-ups \cite[Example 7.2.3 (a)]{thebook} we get an embedded resolution $\C^2 \# 3\overline{\CP^2} \to \C^2$ with graph: 
\[ \ \ \ \ \ 
\xygraph{ 
!{<0cm,0cm>;<1cm,0cm>:<0cm,1cm>::}
!~-{@{-}@[|(2.5)]}
!{(0,0) }*+{\bullet}="x"
!{(-1.5,0) }*+{\bullet}="a1"
!{(1.5,0) }*+{\bullet}="c1"
!{(0,-1.5) }*+{\bullet}="b1"
!{(0,0.5) }*+{E_3}
!{(-1.8,0.5) }*+{E_1}
!{(1.8,0.5) }*+{E_2}
!{(0.5,-1.3) }*+{\widetilde{C}}
"x"-"c1"
"x"-"a1"
"x"-"b1"
} \  .
\]  
Here $E_1^2=-3, E_2^2=-2$, and $E_3^2=-1$. Furthermore the curves $E_1, E_2$, and $E_3$ have multiplicity $m_1=2, m_2=3$ and $m_3=6$ respectively.
\end{exa}

Note that once we erase the unframed vertex from the embedded resolution graph of a curve singularity we get a negative-definite plumbing tree representing $S^3$. (This is because an embedded resolution $\C^2 \# n\overline{\CP^2} \to \C^2$ is in particular a resolution of the trivial surface singularity $(\C^2,0)$.)

The resolution graph $\Gamma$  gives rise to a surgery diagram representing the algebraic knot $(S^3,K)$ associated to the curve singularity $(C,0)\subset (\C^2, 0)$. 
Indeed, given any rooted tree $(\Gamma, v_0)$ and a weight assignment $m: \Gamma \setminus \{v_0\} \to\Z$, we can look at the plumbed three-manifold $Y(G)$ associated to $G=\Gamma \setminus \{v_0\}$ and consider the knot  $(Y(G), K)$ represented by the unframed vertex as in Figure \ref{trefoil}. This gives rise to an interesting class of knots.

\begin{figure}[t]
\begin{center}
\includegraphics[scale=0.8]{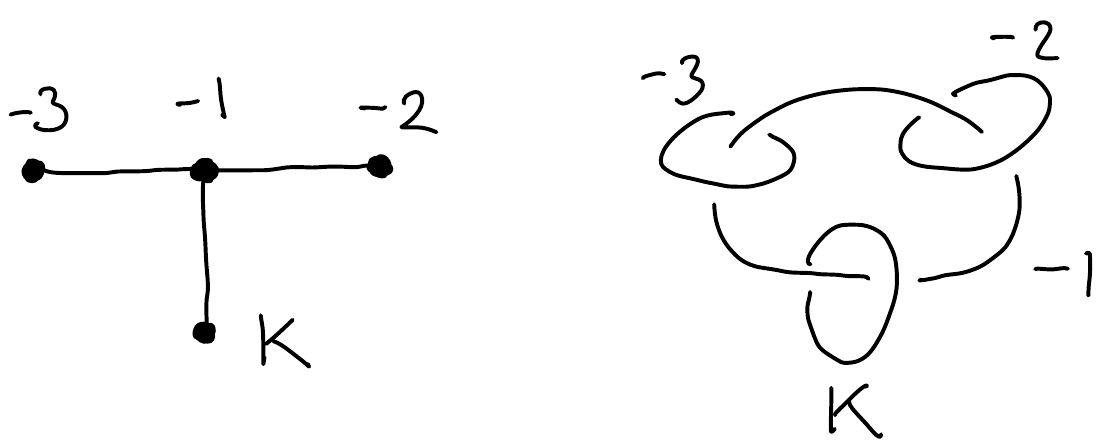}
\caption{\label{trefoil} A surgery diagram of the trefoil knot.}
\end{center}
\end{figure}

\begin{defi}
A knot $(Y, K)$ that can be presented by means of a plumbing tree $\Gamma$ with one unframed vertex $v_0$ is called a graph knot.    
\end{defi}

\begin{exa} Let $(X, 0) \subset \C^N$ be a complex surface singularity and $(C,0) \subset (X,0)$ be a complex curve singularity. Then $(Y,K)=(S_\epsilon(0) \cap X, S_\epsilon(0) \cap C)$ is an algebraic knot in the sense of this new definition. 
\end{exa} 

\begin{exa}Let $(Y_1, K_1)$ and $(Y_2, K_2)$ are algebraic knots with plumbing diagrams $\Gamma_1$ and $\Gamma_2$. Then the connected sum $(Y_1\# Y_2, K_1\#K_2)$ is an algebraic knot. Indeed $\Gamma=\Gamma_1 * \Gamma_2$, the graph obtained joining $\Gamma_1$ and $\Gamma_2$ along their unframed vertices, gives rise to a diagram for $(Y_1\# Y_2, K_1\#K_2)$.
\end{exa}

\subsection{Branched coverings} A knot $(S^3, K)$ gives rise to a knot in a rational homology sphere $(\Sigma(K), \widetilde{K})$ where the ground three-manifold is given by the branched double cover $\Sigma(K)$ of $S^3$ along $K$, and $\widetilde{K}=\text{Fix}(\tau)$ by the fixed point set $\text{Fix}(\tau)$ of the covering involution $\tau: \Sigma(K) \to \Sigma(K)$.  

\begin{prop} \label{procedure}The double branched cover $(\Sigma(K), \widetilde{K})$ of an algebraic knot $(S^3, K)$ associated to a plane curve singularity $(C, 0) \subset (\C^2, 0)$ is algebraic.
\end{prop}    
\begin{proof} Let $ f(x,y)=0$ be an equation for $C\subset \C^2$, and $\rho: \C^2 \# n\overline{\CP^2} \to \C^2$ denote an embedded resolution. Furthermore, let $\rho^{-1}(0)=E_1 \cup \dots \cup E_n \cup \widetilde{C}=E$ denote the exceptional divisor of the embedded resolution.

First we note that $\phi=f\circ \rho$ defines an equation for $E$. This assigns multiplicities $m_1, \dots, m_n \geq 1$ to the rational components of the exceptional divisor. By looking at the branched double cover of $\C^2 \# n\overline{\CP^2}$ along the Weil divisor  $D=\widetilde{C} + \sum_{i=1}^n m_i E_i$ (see \cite[pp. 239-241]{thebook} ) we obtain a smooth surface $\widetilde{S}$ representing a resolution of the surface singularity $(S,0)$ with equation $z^2=f(x,y)$.

The various components of the exceptional divisor $E \subset \C^2 \# n\overline{\CP^2}$ lift to $\widetilde{S}$ as explained in \cite[p. 252]{thebook}  and form a configuration of curves inside $\widetilde{S}$. (Note that for the purposes of singularity theory one usually ignores the pull-back of the strict transform $\widetilde{C}$ to $\widetilde{S}$, while locating the latter here plays a crucial role.) The adjacency graph of this configuration gives rise to a tree $\Gamma$ describing $(\Sigma(K), \widetilde{K})$.   
\end{proof}

\begin{exa}Following the procedure outlined in the proof of Proposition \ref{procedure} we get a diagram for the branched double cover of the resolution graph of the trefoil knot: 
\[ \ \ \ \ \ 
\xygraph{ 
!{<0cm,0cm>;<1cm,0cm>:<0cm,1cm>::}
!~-{@{-}@[|(2.5)]}
!{(0,1.5) }*+{\bullet}="d1"
!{(0,0) }*+{\bullet}="x"
!{(-1.5,0) }*+{\bullet}="a1"
!{(1.5,0) }*+{\bullet}="c1"
!{(0,-1.5) }*+{\bullet}="b1"
!{(0.5,0.5) }*+{-2}
!{(0.5,1.5) }*+{-1}
!{(-1.5,0.5) }*+{-3}
!{(1.8,0.5) }*+{-3}
!{(0.5,-1.3) }*+{K}
"x"-"c1"
"x"-"a1"
"x"-"b1"
"x"-"d1"
} \  .
\]  
This represents a knot in the lens space $L(3,2)$:
\[
\xygraph{ 
!{<0cm,0cm>;<1cm,0cm>:<0cm,1cm>::}
!~-{@{-}@[|(2.5)]}
!{(0,1.5) }*+{\bullet}="d1"
!{(0,0) }*+{\bullet}="x"
!{(-1.5,0) }*+{\bullet}="a1"
!{(1.5,0) }*+{\bullet}="c1"
!{(0.5,1.5) }*+{-1}
!{(0.5,0.5) }*+{-2}
!{(-1.5,0.5) }*+{-3}
!{(1.5,0.5) }*+{-3}
"x"-"c1"
"x"-"a1"
"x"-"d1"
} =
\xygraph{ 
!{<0cm,0cm>;<1cm,0cm>:<0cm,1cm>::}
!~-{@{-}@[|(2.5)]}
!{(0,0) }*+{\bullet}="x"
!{(-1.5,0) }*+{\bullet}="a1"
!{(1.5,0) }*+{\bullet}="c1"
!{(0,0.5) }*+{-1}
!{(-1.5,0.5) }*+{-3}
!{(1.5,0.5) }*+{-3}
"x"-"c1"
"x"-"a1"
} =
\xygraph{ 
!{<0cm,0cm>;<1cm,0cm>:<0cm,1cm>::}
!~-{@{-}@[|(2.5)]}
!{(0,0) }*+{\bullet}="x"
!{(-1.5,0) }*+{\bullet}="a1"
!{(0,0.5) }*+{-2}
!{(-1.5,0.5) }*+{-2}
"x"-"a1"
} = L(3,2) \ .
\]  

Similar computations can be run for all torus knots starting from the defining equation $x^p+y^q=0$.
\end{exa} 

\section{$t$-Modified knot Floer homology, and the upsilon invariant of knots in rational homology spheres}\label{analytictheory}
In \cite{OSS4} Ozsv\'ath, Stipsicz and Szab\' o introduced a one-parameter family of knot homologies $\Kt(K)$ giving rise to  knot invariants of knots in $S^3$. In what follows we go through the straightforward generalisation of their construction taking into account knots in rational homology spheres.

\subsection{Notation}In what follows we will work over the ring $\mathcal{R}$ of long power series with $\F$ coefficients. This is the commutative ring of infinite formal sums  $\sum_{\alpha \in A} \W^\alpha$, with $A \subset \R_{\geq0}$ well-ordered. One defines:
\[\left( \sum_{\alpha \in A} \W^\alpha \right) + \left( \sum_{\beta \in B} \W^\beta \right)= \sum_{\gamma \in A\cup B} \W^\gamma   \]
and 
\[\left( \sum_{\alpha \in A} \W^\alpha \right) \cdot \left( \sum_{\beta \in B} \W^\beta \right)= \sum_{\gamma \in A+ B} c_\gamma \cdot \W^\gamma  \ , \]
where $A+B=\{\alpha+ \beta \ | \ \alpha \in A, \ \beta \in B\} \subset \R_{\geq 0}$ and 
\[c_\gamma=\#\left\{(\alpha, \beta)\in A \times B \ | \ \alpha+\beta=\gamma \right\} \ \mod 2  . \]
The ring $\mathcal{R}$ has the fundamental property that every finitely generated $\mathcal{R}$-module $M$ is sum of cyclic modules  \cite[Section 11]{Brandal} \text{i.e.}
\[ M \simeq \mathcal{R}^k \oplus \mathcal{R}/f_1 \oplus \dots \mathcal{R}/f_m \]
for some $f_1 , \dots , f_m \in \mathcal{R}$, and $k \geq 0$ (the \textit{rank} of $M$). Notice that the field of fraction of the ring $\mathcal{R}$ is given by 
\[ \mathcal{R}^*= \left\{\left.\sum_{\alpha \in A} \W^\alpha \ \right| \ A \subset \R \text{ well-ordered} \right\} \ ,\]
and that the rank of a finitely generated $\mathcal{R}$-module $M$ equals the dimension of  $M_{\mathcal{R}^*}= M \otimes_\mathcal{R} \mathcal{R}^*$ as an $\mathcal{R}^*$-vector space. 

We will think $\mathcal{R}$ as a graded ring, with $\text{deg }\W=-1$. Note that  $\F[U]\hookrightarrow \mathcal{R}$ via the identification $U=\W^2$. 

\subsection{$t$-Modified Knot homologies}
Let $Y$ be a rational homology sphere. Recall \cite{OS7} that a knot $(Y,K)$ can be represented by a doubly-pointed Heegaard diagram $(\Sigma, \boa, \bob, z, w)$. In the symmetric product $\Sym^g(\Sigma)$, the space of degree $g$ divisors over the genus $g$ Riemann surface $\Sigma$, this specifies two half-dimensional, totally-real submanifolds $\T_{\alpha}= \alpha_1 \times \dots \times \alpha_g$ and $\T_\beta= \beta_1 \times \dots \times \beta_g$, and two analytic submanifolds $V_z=\{ z\} \times \Sym^{g-1}(\Sigma)$ and  $V_w=\{ w\} \times \Sym^{g-1}(\Sigma)$ of complex codimension one. We define
\[ CF(\Tt_\alpha, \Tt_\beta)= \bigoplus_{\x \in \Tt_\alpha \cap \Tt_\beta} \mathcal{R} \cdot \x \ .\]   
Given two intersection points $\x$ and $\y \in \Tt_\alpha \cap \Tt_\beta$ consider the set $\pi_2(\x,\y)$ of homotopy classes of topological disks $u: D^2\simeq [0,1] \times\R \to \Sym^g(\Sigma)$  such that 
\begin{itemize}
\item  $u \left( 0\times \R \right) \subseteq \T_{\boa}$ and $u \left(  1 \times \R \right) \subseteq \T_{\bob}$,
\item  $\lim_{t\to - \infty } u(s+it)= \x$ and  $\lim_{t \to + \infty } u(s+it)= \y$.
\end{itemize}
For a generic choice of a path of almost-complex structures $J_s$ we can look at the moduli space $\mathcal{M}(\phi)$ of solutions of the (perturbed) Cauchy-Riemann equation
\begin{equation} \label{cauchyriemann}
\frac{\partial u}{\partial s} (s,t) + J_s \left( \frac{\partial u }{ \partial t} (s,t) \right) = 0
\end{equation} 
within a given homotopy class $\phi \in \pi_2(\x,\y)$ as it was done in \cite{OS2}. It turns out that if we restrict our attention to those classes with Maslov index $\mu(\phi)=1$ then $\mathcal{M}(\phi)$ is a \textit{finite} collection of lines (Gromov's Compactness Theorem). In \cite[Section 4]{OS2} this led to the definition of a differential 
\[\partial \x=\sum_{\y \in \Tt_\alpha \cap \Tt_\beta} \sum_{ \substack{ \phi \in \pi_2(\x , \y) \\    \mu( \phi ) =1  }}\#\left|\frac{\mathcal{M}(\phi)}{\R}\right| \ \W^{2n_z(\phi)} \cdot \y  \]
turning $CF(\Tt_\alpha, \Tt_\beta)$ into a chain complex. Here $n_z(\phi)=\#|\phi(D^2)\cap V_z|$ denotes the intersction with the divisor $V_z$. Notice that the differential of $CF(\Tt_\alpha, \Tt_\beta)$ completely ignores the base point $w$. In fact, the chain homotopy type of $CF(\Tt_\alpha, \Tt_\beta)$ only provides an invariant of the background three-manifold \cite[Theorem 1.1]{OS2}. 

In order to take into account the base point $w$ and hence the knot $K \subset Y$, we can use the following differential 
\[\partial_t \x=\sum_{\y \in \Tt_\alpha \cap \Tt_\beta} \sum_{ \substack{ \phi \in \pi_2(\x , \y) \\    \mu( \phi ) =1  }}\#\left|\frac{\mathcal{M}(\phi)}{\R}\right| \ \W^{tn_z(\phi)+(2-t)n_w(\phi)} \cdot \y \ \ \  \ \ \text{for } t \in [0,2] \ ,  \]
also recording the intersection with the divisor $V_w$. We will denote by $\Kt(Y,K)$ the homology of the resulting chain group  $CF_t(\Tt_\alpha, \Tt_\beta)=(CF(\Tt_\alpha, \Tt_\beta), \partial_t)$.

\subsection{$\Spinc$ refinement}  
In \cite[Section 2.6]{OS2} Ozsv\' ath and Szab\' o built a map $\s_z: \Tt_\alpha \cap \Tt_\beta \to \Spinc(Y)$ associating to an intersection point a $\Spinc$ structure of $Y$. Define
$CF_t(\Tt_\alpha, \Tt_\beta, \s)=\bigoplus_{\substack{  \s_z(\x)=\s}} \mathcal{R} \cdot \x \ .$
Since \cite[Lemma 2.19]{OS2} for any pair of intersection points $\pi_2(\x, \y)$ is non-empty iff $\s_z(\x)=\s_z(\y)$,  we conclude that $CF_t(\Tt_\alpha, \Tt_\beta, \s)$ is a sub-complex of  $CF_t(\Tt_\alpha, \Tt_\beta)$. We set $\Kt(Y, K, \s)=H_*(CF_t(\Tt_\alpha, \Tt_\beta,\s))$. 
 
\subsection{Gradings} Attached to an intersection point $\x \in \Tt_\alpha \cap \Tt_\beta$ of a doubly pointed Heegaard diagram there are two rational numbers: the Alexander grading $A(\x)$ and the Maslov grading $M(\x)$. These have the property that 
\[A(\x)-A(\y)=n_w(\phi) - n_z(\phi) \ \ \  \text{ and } \ \ \ M(\x)-M(\y)= \mu(\phi) - 2n_z(\phi) \]
where $\phi \in \pi_2(\x, \y)$ is disk connecting $\x$ to $\y$. We define a real-valued grading $\gr_t$ on $CF_t(\Tt_\alpha, \Tt_\beta)$ via the formula $\gr_t(\x)=M(\x)-tA(\x)$. Note that $\partial_t$ drops the grading by one \cite[Lemma 3.3]{OSS4}. We set
\[\Upsilon_{K ,\s}(t)= \max \{ \gr_t(\xi) \ | \ \xi \in \Kt(Y,K, \s) \text{ with } \W^\alpha \cdot \xi \not=0 \text{ for } \alpha>0 \} \ . \]  
This is the upsilon invariant of the knot $(Y,K)$ in the $\spinc$ structure $\s$. In the sequel we list some basic properties of the upsilon invariant. 

\begin{prop}Let $(Y,K)$ be a knot in a rational homology sphere and $\s$ a $\spinc$ structure of $Y$, then
\begin{itemize}
\item $\Upsilon_{K,\s}(0)=d(Y, \s)$ where $d(Y, \s)$ denotes the Heegaard Floer correction term of the pair $(Y,\s)$ as defined by Ozsv\' ath and Szab\' o in \cite{OS24},
\item $\Upsilon_{K,\s}(t)=\tau_s(K) \cdot t +d(Y, \s)$ for all $t>0$ close enough to zero, where $\tau_s(K)$ denotes the $\tau$-invariant defined by Grigsby, Ruberman, and Strle in \cite{grigsby2008knot},
\item the invariant $\Upsilon_{K, \s}(t)$ defined presently agrees with the one of \cite{upsiloncovers}, that is
\[\Upsilon_{K, \s}(t)= -2 \cdot \min_\xi \left\{ \frac{t}{2} A(\xi) + \left( 1- \frac{t}{2}\right) j(\xi) \right\}+d(Y(G), \s) \ ,\] 
where $\xi$ ranges between  all cycles with Maslov grading $d= d(Y(G), \s)$ 
\end{itemize} 
Furthermore, $\Upsilon_{K,\s}(t)$ is an invariant of $\Spinc$ rational homology concordance.
\end{prop}
\begin{proof}
The first assertion follows from the fact that $\partial_t$ agrees with the differential of $CF^-(Y,\s)$ when $t=0$. The fact that $\Upsilon_{K,\s}(t)=\tau_s(K) \cdot t +d(Y, \s)$ for small values of $t$ follows from the same argument of \cite[Proposition 1.6]{OSS4}. While the last assertion is proved with the same argument presented in \cite[Section 14.1]{Livingston1}. 

The fact that $\Upsilon_{K,\s}(t)$ is an invariant of $\Spinc$ rational homology concordance was proved in \cite[Proposition 4.1]{upsiloncovers}.
\end{proof}  

\subsection{Zemke's inequality} A key feature of the Heegaard Floer correction terms is the Ozsv\' ath and Szab\' o inequality \cite{OS24} relating the correction terms of two three-manifold connected by a negative-definite $\spinc$ cobordism. This asserts that given a $\spinc$ cobordism $(W, \mathfrak{t}): (Y_0, \s_0) \to (Y_1, \s_1)$ between two $\spinc$ rational homology spheres $(Y_0, \s_0)$ and $(Y_1, \s_1)$ with $b_1(W)=b_2^+(W)=0$ one has that:
\[d(Y_1, \s_1) \geq d(Y_0, \s_0) + \frac{c_1(\mathfrak{t})^2+ b_2(W)}{4} \ .\]
   
In \cite{zemkegradings} Zemke proved that a similar inequality holds for the upsilon invariant.

\begin{thm}[Zemke \cite{zemkegradings}]\label{Zemkeinequality} Let $(Y_0, K_0)$ and $(Y_1, K_1)$ be two knots, and $\s_i\in \Spinc(Y_i)$ $\spinc$ structures. Suppose that there is a $\spinc$ cobordism $(W, \mathfrak{t}): (Y_0, \s_0) \to (Y_1, \s_1)$  containing a properly embedded surface $\Sigma\hookrightarrow W$ with $\partial \Sigma= K_1\cup-K_0$. If $b_1(W)=b_2^+(W)=0$ then 
\[\Upsilon_{K_1, \s_1}(t)\geq \Upsilon_{K_0, \s_0}(t)+ \frac{c_1(\mathfrak{t})^2+ b_2(W)-2t\langle c_1(\mathfrak{t}), [\Sigma]\rangle +2t [\Sigma]^2 }{4} + g(\Sigma) \cdot (|t-1|-1) ,\]
where $g(\Sigma)$ denotes the genus of the surface $\Sigma$.  
\end{thm}

Suppose that $(Y,K)$ is an algebraic knot represented by a negative definite plumbing tree $\Gamma$. Let $v_0 \in \Gamma$ be the unframed vertex and $G=\Gamma\setminus \{v_0\}$. Then $K \subset Y=Y(G)$ bounds a smooth disk $\Delta \subset X(G)$ properly embedded in the plumbing of spheres associated to $G$. (If $(Y,K) = (S^3, K)$ is the link of a plane curve singularity $ (C, 0) \subset (\C^2, 0)$ then $X(G)$ is identified with the total space of a resolution $\rho: \C^2 \# n\overline{\CP^2} \to \C^2$ of $\C^2$ at the origin and $\Delta = \widetilde{C}$ is just the proper transform of $C$.) 

Since $X(G)$ is simply connected, given a $\spinc$ structure $\s$ of $Y(G)$ we can chose an extension $\mathfrak{t}$ to $X(G)$. Then according to Theorem \ref{Zemkeinequality} one has that 
\begin{equation}\label{charineq} \Upsilon_{K, \s}(t)\geq  \frac{k^2+ |G|}{4}-t \cdot \frac{ k \cdot F- F^2 }{2} \ ,
\end{equation} 
where $k=c_1(\mathfrak{t})$ denotes the first Chern class of $\mathfrak{t}$, and $F \in H_2(X(G), \Q)$ a homology class representing the Poincar\' e dual of  $\phi(c)=-\#(\Delta \cap c)$, $c \in H_2(X(G), \Z)$. 
See Section \ref{alexanderfiltration} below. 

In what follows we will show (Theorem \ref{maingoal}) that the inequality displayed in Equation~\eqref{charineq} is sharp, that is $\Upsilon_{K, \s}(t)= (k^2+ |G|)/4-t \cdot (k \cdot F- F^2)/2$ for some  characteristic vector $k$, if the graph $G$ satisfies suitable combinatorial hypothesis.

\section{Deformations of lattice cohomology}
\subsection{A quick review of lattice cohomology} Let $G$ be a negative-definite plumbing of spheres. Denote by $X(G)$ the plumbing of spheres associated to $G$. Let 
\begin{align*}
\text{Char}(G)&=\{c_1(\s)  : \s \in \Spinc(X(G))\} \\
&=\{ k \in H^2(X(G), \Z)  :   k(x) \equiv x^2 \ (\text{mod 2}) \text{ for every }x \in H_2(X(G), \Z)\} \ .
\end{align*}
be the set of characteristic vectors of the intersection form of $X(G)$. Notice that, since $X(G)$ is simply connected, a $\Spinc$ structure $\s$ of $X(G)$ is uniquely determined by its first Chern class $c_1(\s)$. Thus, $\text{Char}(G)\simeq \Spinc(X(G))$. 

In what follows we will be interested in the $\Spinc$ structures of $Y(G)=\partial X(G)$. These always extend over $X(G)$, and two $\Spinc$ structures represented by characteristic vectors $k$ and $k'$ induce the same $\Spinc$ structure on the boundary $\partial X(G)=Y(G)$  if and only if $k-k' \in 2 \cdot H^2(X(G), Y(G))\simeq 2 \cdot H_2(X(G))$.

In \cite{OS20} a computational scheme for the Heegaard Floer homologies of graph manifolds was described. This eventually lead to the definition of  the \textit{lattice homology groups} \cite{Nemethi1} whose construction we presently review. Let $s$ denotes the number of vertices of $G$, and $\s$ a $\Spinc$ structure of $Y$. Think $H_2(X(G), \Z)= \Z^s$ as a lattice in $H_2(X(G) , \R)= \R^s$. The points of $\Z^s \subset \R^s$ specify the vertices of  a subdivision into hypercubes of $H_2(X(G) , \R)=\R^s$, and hence a CW-complex decomposition of $\R^s$. A $p$-cell of this CW-complex decomposition is specified by a pair $(\ell, I)$ where $\ell \in H_2(X(G), \Z)=\Z^s$, and $I \subset G$ with $|I|=p$. More specifically, we associate to such a pair $(\ell, I)$ the $|I|$-cell corresponding to the convex hull of $\{ \ell + \sum_{v \in J} v \ |\  J \subset I \}$. Fix a reference characteristic vector $k\in \Char(G)$ representing $\s$, and set 
$ \chi_k(x) =- \frac{1}{2} (k(\ell)+ \ell^2)$. This is the Riemann-Roch quadratic form associated to the characteristic vector.  
For a $p$-cell $\Box$ of the latter CW decomposition of $\R^s$ we set 
\[w_k(\Box)= \max_{\text{vertices of } \Box} \chi_k(v) \ . \]
For $l \in \Z$ consider the sub-level set $M_l= \bigcup_{w_k(\Box)\leq l}\Box$  and form the chain complex
\[ \CF_*(G,\s)=\bigoplus_{l \in \Z} C_{*}(M_l, \F)  \ ,\]
where $C_*(-, \F)$ denotes CW homology with $\F$-coefficients. (Note that sub-level sets are sub-complexes of $\R^s$ since $w_k(\Box_i)\leq w_k(\Box)$ for every $(p-1)$-dimensional face $\Box_i$ of a given $p$-cell $\Box$.) This is the lattice homology chain group associated to $(G, \s)$. Notice that the inclusions $\dots M_{l-1} \hookrightarrow M_l \hookrightarrow M_{l+1} \dots$ induce a chain map $U: \CF_*(G,\s) \to \CF_*(G,\s)$ turning $\CF_*(G,\s)$ into a chain complex over the ring $\F[U]$. Notice that $\CF_*(G,\s)$ has a natural $\F[U]$-basis: $\mathcal{B}=\{ \Box \in  C_*(M_{w_k(\Box)},\Z) : \Box \  p\text{-face of } \R^s , \ 0 \leq p \leq s\}$. With respect to $\mathcal{B}$ the differential of  $\CF_*(G,\s)$ reads as:
\begin{equation} \label{latticedifferential}
\partial \Box = \sum_{ i}  U^{w(\Box)- w (\Box_i)} \cdot \Box_i \ ,
\end{equation}
where the sum is extended to all $(p-1)$-dimensional faces $\Box_i$ of $\Box$.

\subsection{Gradings} 
In addition to the grading induced by the homological grading of the $C_{*}(M_t, \Z)$ summands, $\CF_*(G,\s)$ has another  grading corresponding to the Maslov grading of the analytic theory: 
\[ \ \ \  \ \ \  \ \ \  \ \ \ \gr(\Box)=\dimm(\Box)-2w_k(\Box)+ \frac{k^2+ |G|}{4} \ ,\] 
for a basis element $\Box \in \mathcal{B}$. This is then extended to $\F$-generators via the identity $\gr(U^j\cdot\Box)=\gr(\Box)-2j$. 

\subsection{The filtration of an algebraic knot}\label{alexanderfiltration}
An algebraic knot $K \subset Y(G)$ is a knot that can be described by a plumbing tree $\Gamma$ with one unframed vertex $v_0$ such that $G=\Gamma -v_0$. As in the analytic theory, the choice of such a knot $K$ induces a filtration on $\CF_*(G, \s)$. We define the Alexander grading of a generator $\Box \in\mathcal{B}$ by
\[A(\Box)= w_{k+2v_0^*}(\Box)- w_k(\Box) + \frac{k \cdot F -F^2}{2} \ .\]
where $F \in H_2(X(G), \Q)$ is a rational homology class representing the Poincar\' e dual\footnote{The class of $v_0$ makes no sense in $H_2(X(G); \Z)$ since $v_0$ does not represent a closed surface in $X(G)$. On the other hand, $v_0$ represents a properly embedded disk $\Delta \subset X(G)$ and hence a class in $H_2(X(G), \partial X(G); \Z)$. Thus we can consider its Poincar\' e dual $v_0^*\in Hom(H_2(X(G); \Z), \Z)$. This is characterized by the property that  $v_0^* \cdot v=1$ if $v_0v$ is in $\Gamma$, and zero otherwise.} of $-v_0^*$, \text{i.e.} such that $F\cdot v = -v_0^* \cdot v$ for each $v \in G$. We define the Alexander grading of a chain $\xi = \sum_{j=1}^m  U^{m_j}\Box_j$ as the maximum of the Alexander grading of its components $A (\xi)= \max_j A(\Box_j)-m_j$. Note that the multiplication by $U$ drops the Alexander grading by one.

\begin{prop} $A(\partial \Box)\leq A(\Box)$.
\end{prop}
\begin{proof}We have to prove that for a $p$-cell $\Box$ we have
\[A( U^{w_k(\Box)- w_k(\Box_i)} \cdot \Box_i ) \leq A(\Box)\]
for every $(p-1)$-face $\Box_i$ of $\Box$. Since multiplication by $U$ drops the Alexander grading by one everything boils down to prove that
\[ A(  \Box_i ) - (w_k(\Box)- w_k (\Box_i)) \leq A(\Box) \ .\]
Substituting the value of the Alexander filtration, and canceling $(k \cdot F +F^2)/4$ on both sides we get  
\[ w_{k+2v_0^*}(\Box_i)-\cancel{ w_k(\Box_i) }  - (\cancel{w_k(\Box)}- \cancel{ w_k (\Box_i)})  \leq  w_{k+2v_0^*}(\Box)- \cancel{w_k(\Box )} \ , \]
On the other hand the inequality $w_{k+2v_0^*}(\Box_i) \leq  w_{k+2v_0^*}(\Box)$ follows immediately from the definitions.   
\end{proof}

\begin{rmk}
Exactly as in the analytic theory the group $\CF(G, \s)$ has an algebraic filtration $j$. For a $q$-cell $\Box \subset M_l$ this is given by $j(\Box)= w_k(\Box)- l$. 
\end{rmk}

\subsection{Proof of Theorem \ref{maingoal} for rational graphs} Recall that a vertex $v$ of a negative definite plumbing tree $G$ is said to be \textit{bad} if $\deg (v) >- v^2$. Using the number of bad vertices we can partition algebraic knots into complexity classes: we say that a knot $K\subset Y(G)$ is of \textit{type-$k$} if it can be represented by a plumbing diagram $\Gamma$ with underlying plumbing tree $G$ having no more than $k$ bad points. 

We now prove that \eqref{charineq} is sharp in the special case of algebraic knots that can be represented by means of a plumbing diagram with no bad points. 

\begin{prop} Suppose that $K$ is an algebraic knot of type-0 (no bad vertices) then 
\begin{equation}\label{inspiring}
 \Upsilon_{K, \s}(t)=  -2 \min_{x\in \Z^s} \chi_t(x)+\left( \frac{k^2+ |G|}{4} - t \ \frac{k\cdot F-F^2}{2} \right) \ , \end{equation}
where $\chi_t$ denotes the twisted Riemann-Roch function  
\[\chi_t(x)=-\frac{1}{2}\left( (k+tv_0^*) \cdot x + x^2 \right)\ .\]
\end{prop}  
\begin{proof} According to \cite{Lspaceknots}  $CFK^\infty(K, Y(G), \s)$ has the same filtered chain homotopy type of $(\CF^-(G, \s) \otimes_{\F[U]} \F[U, U^{-1}], \partial, A)$. Thus
\begin{equation}\label{pippo}
\Upsilon_{K, \s}(t)= -2 \cdot \min_\xi \left\{ \frac{t}{2} A(\xi) + \left( 1- \frac{t}{2}\right) j(\xi) \right\}+d(Y(G), \s) 
\end{equation} 
where the minimum is taken over all cycles $\xi$ with Maslov grading $d= d(Y(G), \s)$ representing the generator of $H_d(\CF^-(G, \s) \otimes_{\F[U]} \F[U, U^{-1}])=\F$. 

Since $H_p(\R^s, \Z)=0$ for $p > 0$ any $p$-cycle $\xi \subset \R^s$ eventually bounds a $(p+1)$-chain in $M_{l}$. Hence the minimum in Equation~\eqref{pippo} can be taken over all cycles of the form $\xi=U^{-j} \cdot x$, with $x$ representing a vertex of the CW-decomposition of $\R^s$, and $j \in \Z$. Imposing $\gr(U^{-j} \cdot x)=d$, and substituting into Equation \eqref{pippo}, we get 
\begin{align*}
\Upsilon_{K, \s}(t) &= -2 \cdot \min_{x \in \Z^s}\left\{ 
\frac{t}{2}\left(A(x)+ \frac{d-\gr(x)}{2}\right) + \left(1-\frac{t}{2} \right) \left( \frac{d-\gr(x)}{2}\right)\right\}+ d\\ 
&=  -2 \cdot \min_{x \in \Z^s}\left\{ \frac{t}{2}\big(\chi_{k+2v_0*}(x) -\chi_k(x) \big) +  \chi_k(x) + t \ \frac{k\cdot F- F^2}{4} -\frac{k^2+ |G|}{8}\right\} \\
&= -2 \cdot \min_{x \in \Z^s}\left\{ -\frac{t}{2}v_0^* \cdot x +   \chi_k(x)\right\} +\frac{k^2+ |G|}{4} - t \ \frac{k\cdot F- F^2}{2} \ ,
\end{align*} 
from where the claimed identity. 
\end{proof}  

\subsection{Constructions of the groups}\label{construction}
Formula \eqref{inspiring} suggests that the upsilon invariant of an algebraic knot can be expressed as the correction term of  suitable lattice groups. 

Let $K \subset Y(G)$ be an algebraic knot ($G$ negative-definite) presented by a tree $\Gamma$ with one unframed vertex $v_0$. For $t \in [0,2]$ we twist the Riemann-Roch function by means of the real cohomology class $t v_0^* \in H^2(X(G), \R)$ \[\chi_t(x)=-\frac{1}{2}\left( (k+tv_0^*) \cdot x + x^2 \right) \ . \]   
Again we stress the fact that the homology class of the vertex $v_0$ does not exists while its Poincar\' e dual is always defined. With this said, we extend $\chi_t$ to $p$-cells ($0<p\leq n$) via the identity
\begin{equation}\label{levelfunction}
 w_t(\Box)=(2-t) \cdot w_{k}(\Box)+t \cdot w_{k+2v_0^*}(\Box)   \     .  
\end{equation}

Note that $w_t$ has the key property that $w_t(\Box_i)\leq w_t(\Box)$ for every face $\Box_i$ of $\Box$. For each real parameter $l \in \R$ we can now consider the sub-level set $M_l= \bigcup_{w_t(\Box)\leq l}\Box$ and form the persistent homology module 
\[\CF_t(\Gamma,\s)=\prod_{l \in \R} C_{*}(M_l, \F)  \ .\]
This has an $\mathcal{R}$-module structure: for $\alpha \geq 0$ we set $\W^\alpha \cdot \Box= \iota_{ \#}^\alpha(\Box)$, where 
$\iota^\alpha : M_l \hookrightarrow M_{l+\alpha}$ denotes the inclusion, and we extend it to an $\mathcal{R}$-action by linearity. (Note that the use of direct products in the definition of $\CF_t$ is crucial.) 

\subsection{Gradings in the deformations} As in the analytic setting we define a grading on $\CF_t(\Gamma, \s)$ setting $\gr_t(\Box)= \gr(\Box)- t A(\Box)$. 
\begin{prop} The differential of $\CF_t(\Gamma, \s)$ drops the grading by one.
\end{prop}
\begin{proof}Using $\mathcal{B}=\{ \Box \in  C_*(M_{w_k(\Box)},\Z) : \Box \  p\text{-face of } \R^s , \ 0 \leq p \leq s\}$ as $\F[U]$-basis the differential $\partial_t$ of  $\CF_t(\Gamma,\s)$ writes as 
\begin{equation} 
\partial_t \Box = \sum_{ i}  \W^{w_t(\Box)- w_t (\Box_i)} \cdot \Box_i \ ,
\end{equation}
where the sum is extended to all faces $\Box_i$ of $\Box$. For a $p$-cell $\Box$ one computes 
\[\gr_t(\Box)=\dimm(\Box) -w_t(\Box) + \frac{k^2+|G|}{4}-t \ \frac{k \cdot F- F^2}{2}  \ .\]
Thus, for a component $c_i=\W^{w_t(\Box)- w_t (\Box_i)} \cdot \Box_i$ of the differential $\partial_t \Box$ of a $p$-cell we have 
\begin{align*}
\gr_t(\Box)-\gr_t(c_i)&= \gr_t(\Box)- \gr_t(\Box_i)- (w_t(\Box)- w_t(\Box_i))\\
&= (\dimm(\Box)- \cancel{w_t(\Box)})- (\dimm(\Box_i)- \cancel{w_t(\Box_i)})- \cancel{w_t(\Box)}+ \cancel{ w_t(\Box_i)}\\
&= \dimm(\Box)- \dimm(\Box_i)=1  \ . 
\end{align*}

\vspace{-0.6cm}
\end{proof}

\subsection{Correction terms}\label{combcorterms}Here is a structure theorem for the module $t\mathbb{HFK}_*$. 
\begin{lem}\label{shape}
The group $t\mathbb{HFK}_*(\Gamma, \s)$ is an $\mathcal{R}$-module of rank one. Furthermore, non-torsion elements are concentrated in lattice grading zero.  
\end{lem}
\begin{proof}
Let $\xi \subset M_l$ be a $p$-chain, for some $p>0$. Since $\R^s$ is contractible there exists $l'\geq l$ such that $\xi=\partial \beta$ in $M_{l'}$. Thus $\W^{l'-l} \cdot [\xi]=0$, showing that all cycles of this type are torsion. 

If $p=0$ on the other hand, $\xi=v_1+ \dots+v_m$ with either $m$ even or odd. If $m$ is even then we can find $l'\geq l$ and a $1$-chain $\gamma \subset M_{l'}$ such that either $\xi=\partial \gamma$. Again this proves that  $\W^{l'-l}[\xi]=0$, hence 
that $[\xi]$ is an element of $\mathcal{R}$-torsion. If $m$ is odd on the other hand given any other $0$-cycle $\xi'=v_1+\dots +v_{2n+1} \subset M_{l'}$ we can find $1$-chain $\gamma \subset M_{l''}$ with $l''>\max\{l, l'\}$ such that $\xi-\xi'=\partial\gamma$. In this case $\W^{l''-l} [\xi] + \W^{l''-l'} [\xi']=0$ and we are done showing that the rank is one.    
\end{proof}

As in the analytic theory we define 
\[\Upsilon_{\Gamma, \s}(t)=\max \{ \gr_t(\xi) \ | \ \xi \in t\mathbb{HFK}(\Gamma, \s) \text{ with } \W^\alpha \cdot \xi \not=0 \text{ for } \alpha>0 \} \ .\] 
The following proves that $\Upsilon_{\Gamma, \s}(t)$ can be computed combinatorially as in Equation \eqref{inspiring}. 

\begin{lem}\label{goal}
Let $k$ be a characteristic vector representing the $\Spinc$ structure $\s$. Then 
\[\Upsilon_{\Gamma, \s}(t)=-2 \min_{x\in \Z^s} \chi_t(x)+\left( \frac{k^2+ |G|}{4}- t \ \frac{k\cdot F- F^2}{2} \right)  \ . \] 
\end{lem}
\begin{proof}As consequence of Lemma \ref{shape} we have that
\begin{align*}\Upsilon_{\Gamma, \s}(t)&= \max_{x\in \Z^s} \gr_t(x)\\
&=-\min_{x\in \Z^s}\Big\{(2-t)\chi_k(x)+t\chi_{k+2v_0^*}(x)\Big\} + \left( \frac{k^2+ |G|}{4}- t \ \frac{k\cdot F- F^2}{2} \right)\\
&=-2\min_{x\in \Z^s}\chi_t(x) + \left( \frac{k^2+ |G|}{4}- t \ \frac{k\cdot F- F^2}{2} \right)\ ,
\end{align*} 
and we are done.
\end{proof}

\section{The surgery exact triangle of the $t$-modified knot homologies}\label{sectionanalytictriangle}
Let $\Sigma$ be a genus $g$ Riemann surface, and let $\boe^1, \dots \boe^k$ be collections of compressing circles for some genus $g$ solid handlebodies $U_{\boe_1}, \dots , U_{\boe_k}$ with boundary $\Sigma$. For $i=1, \dots, k$ let $\Tt_i=\eta_i^1\times  \dots \times \eta_i^g \subset \Sym^g(\Sigma)$ denote the Lagrangian torus associated to $\boe^i=\{\eta_i^1, \dots , \eta_i^g\}$. Without loss of generality we can assume that the various $\eta$-curves intersect transversely, hence  $\Tt_i \pitchfork \Tt_j$ for $i\not=j$. Given intersection points $\x_i \in \Tt_{i-1} \cap  \Tt_{i} $ for $i=1, \dots, k$ and $\y \in \Tt_1 \cap \Tt_k$,  denote by $\pi_2(\x_1, \dots, \x_k, \y)$ the set of homotopy classes of continuous maps $u: D^2 \to \Sym^g(\Sigma)$ with domain the complex unit disk $D^2$ with $k+1$ marked points $z_0, z_1, \dots z_k \in \partial D^2$ (lying in the order on the unit circle) such that
\begin{itemize}
\item $u(z_i)=\x_i$ for $i=1, \dots , k$ and $u(z_0)=\y$
\item $u(a_i) \subset \Tt_i$ where  $a_i \subset \partial D^2$ denotes the boundary arc in between $z_i$ and $z_{i+1}$ for $i=0, \dots ,k$ (mod $k$).
\end{itemize}
We will be interested in the moduli spaces $\mathcal{M}(P)$ of pseudo-holomorphic representatives of a given homotopy class $P \in \pi_2(\x_1, \dots, \x_k, \y)$, \text{i.e.} maps  $u: D^2 \to \Sym^g(\Sigma)$ in $P$ solving  the Cauchy-Riemann equation on the interior of the unit disk. Note that for $k \geq 4$ the source of these maps themselves have moduli: if $\mathcal{M}_{0,k}$ denotes the moduli space of disks with $k$ punctures on the boundary then $\dimm \mathcal{M}_{0,k}=k-3$. As in the case of pseudo-holomorphic strips discussed in Section \ref{analytictheory}, associated to an homotopy class $P \in \pi_2(\x_1, \dots, \x_k, \y)$ there is a Maslov index $\mu(P) \in \Z$. For a generic choice of perturbations of the Cauchy-Riemann equation, the moduli space $\mathcal{M}(P)$ forms a smooth finite-dimensional manifold of dimension 
\[\dimm \mathcal{M}(P)= \mu(P)+ \dimm \mathcal{M}_{0,k}=\mu(P) + k-3 \ .\]   
We define maps $ f_{\boe_1, \dots ,\boe_k}: CF_t(\Tt_{k-1}, \Tt_k) \otimes_\mathcal{R} \dots \otimes_\mathcal{R} CF_t(\Tt_{1}, \Tt_2) \to CF_t(\Tt_1, \Tt_k)$ by counting pseudo-holomorphic $k$-gons
\[f_{\boe_1, \dots ,\boe_k}(\x_k \otimes \dots \otimes \x_1)= \sum_{y \in \Tt_0 \cap \Tt_k} \sum_{\mu(P)=3-k} \# \mathcal{M}(P)  \ \W^{tn_z(P)+ (2-t)n_w(P)} \cdot \y \ .\]

An inspection of the ends of moduli spaces of pseudo-holomorphic $k$-gons with Maslov index $\mu=2-k$ shows that the maps $f_{\boe_1, \dots ,\boe_k}$ satisfy the so called $A_\infty$-relations:
\[\sum_{0\leq i < j \leq k}f_{\boe_1, \dots, \boe_i, \boe_j, \dots ,\boe_k}(\x_1\otimes \dots \otimes\x_{i-1}\otimes f_{\boe_i, \dots ,\boe_j}(\x_{i}\otimes \dots \otimes \x_{j-1})\otimes \x_{j}\otimes \dots \otimes \x_k )=0 \ . \]
We will be interested in these relations for low values of $k$. If we set $\x \cdot \y=f_{\boe_1, \boe_2, \boe_3}(\x\otimes \y )$ then the $A_\infty$-relations for $k=4$ read as
\begin{equation}
\partial_t (\x\cdot \y)=\partial_t \x\cdot \y +\x\cdot \partial_t\y \ ,
\end{equation}
proving that $\x \cdot \y$ satisfies the Leibeniz rule. 
Note that this product operation is not associative. On the other hand, for $k=5$ the $A_\infty$-relations say that so happens up to homotopy. More precisely we have that:
\begin{equation}
(\x \cdot \y) \cdot\boldsymbol{z}+ \x \cdot (\y \cdot \boldsymbol{z})= \partial_t f_{\boe_1, \boe_2, \boe_3, \boe_4}(\x \otimes  \y \otimes \boldsymbol{z})+f_{\boe_1, \boe_2, \boe_3, \boe_4}(\partial_t(\x \otimes \y \otimes \boldsymbol{z}))  \ . 
\end{equation}

Suppose now that $K \subset Y$ is a knot in a rational homology sphere and that $C \subset Y\setminus K$ is a framed loop in its complement. We wish to establish an exact triangle of the form
\begin{equation}\label{exactanalytic}
\begin{tikzcd}[column sep=small]
 \Kt(Y,K) \arrow{rr} &     & \Kt(Y_\lambda(C), K) \arrow{ld} \\
& \Kt(Y_{\lambda+\mu}(C), K)  \arrow{lu}   & 
\end{tikzcd} 
\end{equation}
where $\lambda$ denotes the chosen longitude of $C$, and $\mu$ a meridian. To this end we model the triple $(Y, Y_\lambda(C), Y_{\lambda+\mu}(C))$ by means of four collections of compressing circles $\boa, \bob, \bog$ and $\bod$ on a genus $g$ Riemann surface $\Sigma$. More precisely we choose:
\begin{itemize}
\item the $\alpha$- and the $\beta$-curves so that $(\Sigma, \boa, \bob)$ forms a Hegaard diagram of $Y$,
\item the first $\beta$-curve $\beta_1$ to be a meridian $\mu$ of $C$, the first $\gamma$-curve $\gamma_1$ to be the longitude $\lambda$, and the first of the $\delta$-curves to be a curve  of type $\lambda+\mu$,
\item the last $g-1$ $\gamma$- and $\delta$-curves to be small Hamiltonian translates of the corresponding last $g-1$ $\beta$-curves.
\end{itemize}     
Note that the base points $z$ and $w$ can be chosen so that $\mathcal{H}_{\boa, \bob}=(\Sigma, \boa, \bob, z, w)$ represents $(Y, K)$, $\mathcal{H}_{\boa, \bog}=(\Sigma, \boa, \bog, z, w)$ represents $(Y_\lambda(C), K)$, and $\mathcal{H}_{\boa, \bod}=(\Sigma, \boa, \bod, z, w)$ represents $(Y_{\lambda+\mu}(C), K)$. Notice that we can assume: $\beta_2$ to be a meridian of $K$, the two base points $z$ and $w$ to lie near to the two sides of $\beta_2$, and the Hamiltonian isotopies sending $\{\beta_2, \dots, \beta_g\}$ onto $\{\gamma_2, \dots, \gamma_g\}$ and  $\{\delta_2, \dots, \delta_g\}$ to not cross the two basepoints. 

We now introduce a triangle of maps
\begin{equation} 
\begin{tikzcd}[column sep=small]
 CF_t(\Tt_\alpha, \Tt_\beta) \arrow{rr}{F_{\bob, \bog}} &     & CF_t(\Tt_\alpha, \Tt_\gamma) \arrow{ld}{F_{\bog , \bod}} \\
& CF_t(\Tt_\alpha, \Tt_\delta) \arrow{lu}{F_{\bod , \bob}}   & 
\end{tikzcd}
\end{equation}
inducing in homology the maps appearing in \eqref{exactanalytic}. First we observe that, since the two basepoints $z$ and $w$ lie on the same connected component of $\Sigma \setminus \bob \cup \bog$, we have an identification $H_*(CF_t(\Tt_\beta, \Tt_\gamma))=\Lambda^*H_1(T^{g-1}) \otimes \mathcal{R}$. In fact, the same equality holds for  $H_*(CF_t(\Tt_\gamma, \Tt_\delta))$ and $H_*(CF_t(\Tt_\delta, \Tt_\beta))$. Denote by $\Theta_{\beta, \gamma}$, $\Theta_{\gamma, \delta}$ and $\Theta_{\delta,\beta}$ the  cycles descending to the top-dimensional generator of $\Lambda^*H_1(T^{g-1}) \otimes \mathcal{R}$ in $CF_t(\Tt_\beta, \Tt_\gamma)$, $CF_t(\Tt_\gamma, \Tt_\delta)$, and $CF_t(\Tt_\delta, \Tt_\beta)$ respectively. We define $F_{\bob, \bog}: CF_t(\Tt_\alpha, \Tt_\beta) \to CF_t(\Tt_\alpha, \Tt_\gamma)$ by $F_{\bob, \bog}(x)=x \cdot \Theta_{\beta, \gamma}$. Note that $F_{\bob, \bog}$ is a chain map:
\[\partial_t F_{\bob, \bog}(x)=\partial_t(x \cdot \Theta_{\beta, \gamma})=
\partial_t x \cdot \Theta_{\beta, \gamma} + \cancel{x \cdot \partial_t \Theta_{\beta, \gamma}}= \partial_t x \cdot \Theta_{\beta, \gamma}= F_{\bob, \bog}(\partial_t x) \ . \]
Analogously we define chain maps $F_{\bog, \bod}: CF_t(\Tt_\alpha, \Tt_\gamma) \to CF_t(\Tt_\alpha, \Tt_\delta)$ and $F_{\bod, \bob}: CF_t(\Tt_\alpha, \Tt_\delta) \to CF_t(\Tt_\alpha, \Tt_\beta)$ using the top-dimensional generators $\Theta_{\gamma, \delta}$ and $\Theta_{\delta,\beta}$. 

We now consider the triangle of maps induced in homology by $F_{\bob, \bog},F_{\bog, \bod}$ and $F_{\bod, \bob}$. We would like to show that the composition of two consecutive maps is zero. To this end define $H_{ \bob, \bog, \bod}: CF_t(\Tt_\alpha, \Tt_\beta) \to CF_t(\Tt_\alpha, \Tt_\delta)$ by counting pseudo hlomorphic rectangles: $H_{ \bob, \bog, \bod}(x)=f_{ \boa, \bob, \bog, \bod}( x \otimes \Theta_{\beta, \gamma}\otimes \Theta_{ \gamma, \delta})$. In this case the $A_\infty$-relations prescribe the identity
\begin{align*}
F_{\bog, \bod}\circ F_{\bob, \bog}(x)&= (x \cdot \Theta_{\beta, \gamma}) \cdot \Theta_{\gamma, \delta} \\
&=x \cdot (\Theta_{\beta, \gamma} \cdot \Theta_{\gamma, \delta})  \\
&+\partial_t( f_{ \boa, \bob, \bog, \bod}(x \otimes \Theta_{\beta, \gamma} \otimes \Theta_{\gamma, \delta})) + f_{ \boa, \bob, \bog, \bod}(\partial_t x \otimes \Theta_{\beta, \gamma} \otimes \Theta_{\gamma, \delta}) \\ 
&+\cancel{f_{ \boa, \bob, \bog, \bod}(x \otimes \partial_t \Theta_{\beta, \gamma}\otimes \Theta_{\gamma, \delta})}+\cancel{f_{ \boa, \bob, \bog, \bod}(x \otimes \Theta_{\beta, \gamma} \otimes \partial_t \Theta_{\gamma, \delta})} \\
&=x \cdot (\Theta_{\beta, \gamma} \cdot \Theta_{\gamma, \delta}) +\partial_tH_{ \bob, \bog, \bod}(x) + H_{ \bob, \bog, \bod}(\partial_t x) \ .  
\end{align*}
On the other hand: $\Theta_{\beta, \gamma} \cdot \Theta_{\gamma, \delta}=0$ based on the very same neck stretching argument of \cite{}. Hence,
\[F_{\bog, \bod}\circ F_{\bob, \bog}= \partial_t \circ H_{ \bob, \bog, \bod} + H_{ \bob, \bog, \bod}\circ \partial_t \ , \]
showing that the composition $F_{\bog, \bod}\circ F_{\bob, \bog}$ is null-homotopic via the map $H_{ \bob, \bog, \bod}$.
Similarly one defines homotopy equivalences $H_{\bog,\bod, \bob}: CF_t(\Tt_\alpha, \Tt_\gamma) \to CF_t(\Tt_\alpha, \Tt_\beta)$ and $H_{\bod,\bob, \bog}: CF_t(\Tt_\alpha, \Tt_\delta) \to CF_t(\Tt_\alpha, \Tt_\gamma)$ for $F_{\bod, \bob} \circ  F_{\bog, \bod}$, and $F_{\bob, \bog} \circ  F_{\bod, \bob}$.

Finally one would like to show that the triangle of maps induced by $F_{\bob, \bog}, F_{\bog, \bod}$ and $F_{\bod, \bob}$
has trivial homology (exactness). This will be based on the following algebraic lemma.

\begin{lem}\label{algebra}Let $\{A_i\}_{i\in\Z}$ be a collection of chain complexes, and let $\{f_i:A_i\to A_{i+1}\}_{i\in \Z}$ be a collection of chain maps satisfying the following two properties:
\begin{enumerate}
\item\label{condition} $f_{i+1}\circ f_i$ is chain homotopically trivial via a chain homotopy $H_i:A_i \to A_{i+1}$,
\item the map $\psi_i=f_{i+2}\circ H_i+H_{i+1}\circ f_i$ is a quasi-isomorphism. 
\end{enumerate} 
Then $H_*(\text{Cone}(f_i))\simeq H_*(A_{i+2})$, where $\text{Cone}(f_i)$ denotes the mapping cone of $f_i$.  
\end{lem}

\begin{proof}[Proof of Theorem \ref{exactanalyticthm}]
We proceed in analogy with the proof of the exact triangle 
\begin{equation} \label{exacthat}
\begin{tikzcd}[column sep=small]
 \widehat{\text{HFK}}(Y, K) \arrow{rr} &     &  \widehat{\text{HFK}}(Y_\lambda(C),  K) \arrow{ld} \\
& \widehat{\text{HFK}}( Y_{\lambda+\mu}(C), K) \arrow{lu}   & 
\end{tikzcd}
\end{equation}
established by Ozsv\' ath and Szab\' o \cite[Theorem 8.2]{OS7}. See also \cite[Section 2]{OS6}.

To meet precisely the hypothesis of Lemma \ref{algebra} we choose a sequence of Hamiltonian translates $\bob^{(i)}, \bog^{(i)}$ and $\bod^{(i)}$ of the $\beta$-curves, the $\gamma$-curves, and the $\delta$-curves. Set $A_{3i}=CF_t(\T_{\boa},\T_{\bob^{(i)}})$, $A_{3i+1}=CF_t(\T_{\boa},\T_{\bog^{(i)}})$, and $A_{3i}=CF_t(\T_{\boa},\T_{\bod^{(i)}})$, and note that there are obvious identifications $A_{3i}=CF_t(\T_{\boa},\T_{\bob})$, $A_{3i+1}=CF_t(\T_{\boa},\T_{\bog})$, and $A_{3i}=CF_t(\T_{\boa},\T_{\bod})$. Furthermore, we define $f_{3i}=F_{\bob^{(i)},\bog^{(i)}}$, $f_{3i+1}=F_{\bog^{(i)},\bod^{(i)}}$, and $f_{3i+2}=F_{\bod^{(i)},\bob^{(i)}}$. Similarly, we take $H_{3i}=H_{ \bob^{(i)}, \bog^{(i)}, \bod^{(i)}}$, $H_{3i+1}=H_{\bog^{(i)},\bod^{(i)}, \bob^{(i)}}$ and $H_{3i+2}: H_{\bod^{(i)},\bob^{(i)}, \bog^{(i)}}$ so that condition \eqref{condition} of Lemma \ref{algebra} is met.

Writing down the $A_\infty$-relations for $k=5$ (the one coming from the count of pseudo-holomorphic pentagons) we get that
\begin{align*}
0&= f_{i+2}(H_i(x))+H_{i+1}( f_i(x))\\
&+ f_{\boa,\bog^{(i)}, \bod^{(i)}, \bog^{(i+1)}}(x \otimes \Theta_{\bog^{(i)}, \bod^{(i)}} \otimes \cancel{\Theta_{\bod^{(i)}, \bob^{(i)}} \cdot \Theta_{\bob^{(i)}, \bog^{(i+1)}}} )\\
&+ f_{\boa,\bog^{(i)}, \bob^{(i)}, \bog^{(i+1)}}(x \otimes \cancel{\Theta_{\bog^{(i)}, \bod^{(i)}} \cdot \Theta_{\bod^{(i)}, \bob^{(i)}}} \otimes \Theta_{\bob^{(i)}, \bog^{(i+1)}})\\
&+x \cdot H_{\bog^{(i)},\bod^{(i)}, \bob^{(i)}, \bog^{(i+1)}}
(\Theta_{\bog^{(i)}, \bod^{(i)}} \otimes \Theta_{\bod^{(i)}, \bob^{(i)}} \otimes \Theta_{\bob^{(i)}, \bog^{(i+1)}}) \ .
\end{align*}  

Hence, in order to show that the exact triangle holds, we must show that the map
\[\psi_i: x \mapsto x \cdot H_{\bog^{(i)},\bod^{(i)}, \bob^{(i)}, \bog^{(i+1)}}
(\Theta_{\bog^{(i)}, \bod^{(i)}} \otimes \Theta_{\bod^{(i)}, \bob^{(i)}} \otimes \Theta_{\bob^{(i)}, \bog^{(i+1)}})  \]
induces an isomorphism in homology. To this end we observe that specialising $q=0$ in the chain complex $CF_t$ we get the hat version of knot Floer homology $\widehat{CFK}$ (with the real-valued $\text{gr}_t$-grading instead of the usual bi-grading), and that $\psi_i$ is a qusi-isomorphism provided that $\widehat{\psi}_i$ (its restriction to $\widehat{CFK}$) is a quasi-isomorphism. On the other hand, the map $\widehat{\psi}_i$  was shown to be a quasi-isomorphism in Ozsv\' ath and Szab\' o's proof of the exact triangle  \cite[Theorem 8.2]{OS7}.      
\end{proof}

\section{A surgery exact sequence for deformations of lattice cohomology}\label{combinatorialexactsequence}
Let $\Gamma$ be a plumbing graph with one unframed vertex $v_0$. Let $G=\Gamma-v_0$ and $v\in G$ be a vertex that is \textit{not} directly connected to $v_0$. We denote by $G_{+1}(v)$ the graph obtained from $G$ by increasing the weight of $v$ by one, and by $G'(v)$ the graph obtained from $G$ by adding a $(-1)$-framed vertex $e$ connected to $v$. Let $\Gamma_{+1}(v)$ and $\Gamma'(v)$ the graphs obtained similarly from $\Gamma$. These represent knots in $Y(G_{+1}(v))$, $Y(G'(v))$, and $Y(G-v)$.

\begin{thm}\label{latticeexact}There is a long exact sequence of $\mathcal{R}$-modules
\[ \xymatrix{
 \ar[r] & t\mathbb{HFK}_p(\Gamma-v) \ar[r]^{\ \ \  \phi_*  } &  t\mathbb{HFK}_p(\Gamma) \ar[r]^{\psi_* \ \ \ }  & t\mathbb{HFK}_p(\Gamma_{+1}(v))  \ar[r]^{\ \  \delta} & t\mathbb{HFK}_{p-1}(\Gamma) \ar[r] &} \ .\]
\end{thm}

The exact sequence of Theorem \ref{latticeexact} is obtained by applying the Snake Lemma to a short exact sequence 
\begin{equation}\label{shortexact}
\xymatrix{
0 \ar[r] &\mathbb{CF}_t(\Gamma-v)  \ar[r]^{ \ \ \  A_t  } &  \mathbb{CF}_t(\Gamma)  \ar[r]^{B_t \ \ \ }  & \mathbb{CF}_t(\Gamma_{+1}(v))  \ar[r] & 0} 
\end{equation}   
preserving the lattice grading. To describe the maps $A_t$ and $B_t$ appearing in \eqref{shortexact} we resort to some convenient notation introduced by Ozsv\' ath, Stipsicz, and Szab\' o \cite{OSS1}. This has the advantage that no choice of ground characteristic vector is needed in the definition of the differential. Passing from one notation to the other consists in choosing an origin for the affine space of characteristic vectors associated to the plumbing.       

Let $G$ be a negative-definite plumbing tree. Let $v_1, \dots, v_s$ denote the vertices of $G$. For every pair $[K, E]$, with $E\subseteq \{v_1, \dots, v_s\}$, and $K\in H^2(X(G); \Z)$ characteristic, define 
\[2 f_G(K,I)=  \sum_{v\in I} K(v) +\left( \sum_{v\in I} v  \right) \cdot \left( \sum_{v\in I} v  \right) = K\cdot I + I^2 \ ,\] 
and take 
\[g_G[K, E]= \min_{I \subseteq E} f(K,I) \ .\] 
Then $\CF^-(G)$ can be identified with the free $\F[U]$-module formally generated by all such pairs $[K, E]$ with differential 
\[\partial[K,E]= \sum_{v \in E } U^{a_v[K,E]}\otimes [K, E- v] + \sum_{v \in E } U^{b_v[K,E]} \otimes [K+2v^*, E- v] \ ,  \]  
where:
\begin{align*}
a_v[K,E]&= g[K,E-v]-g[K,E] \ , \\ 
b_v[K,E]&= g[K+2v^*,E-v]-g[K,E]+ \frac{K\cdot v+v^2}{2} \ .
\end{align*}
 
If $\Gamma$ is a negative-definite plumbing diagram with one unframed vertex $v_0$ representing a knot in $Y(G)$, $G=\Gamma\setminus\{v_0\}$, then the differential of $\CF_t(\Gamma)=\CF(G) \otimes \mathcal{R}$ is given by
\[\partial_t[K,E]= \sum_{v \in E } q^{a_v(t)}\otimes [K, E- v] + \sum_{v \in E } q^{b_v(t)} \otimes [K+2v^*, E- v] \ ,  \] 
where $a_v(t)= (2-t)a_v[K,E]+t a_v[K+2v_0^*, E]$, and similarly $b_v(t)= (2-t)b_v[K,E]+t b_v[K+2v_0^*, E]$. In this context the $\gr_t$-grading is given by
\[\gr_t[K,E]={}^tg_G[K,E]+|E|+\frac{K^2+|G|}{4}-t \ \frac{K\cdot F-F^2}{2} \ ,\] 
where we set ${}^tg_G[K,E]= (2-t)g_G[K,E]+tg_G[K+2v_0^*,E]$.

Let now $v$ be a vertex of $G$. We define $\psi_v: \mathbb{CF}_t(\Gamma-v) \to  \mathbb{CF}_t(\Gamma)$ by
\[\psi_v[K,E] =\sum_{p \equiv v^2 (mod \ 2)}[K, p, E]\]
where  $[K, p, E]$ stands for the generator of $\mathbb{CF}_t(\Gamma)$ associated to the subset $E$, and the characteristic vector $(K,p) \in \Char(G)$ extending $K \in \Char(G-v)$ by $K(v)=p$. (The fact $p \equiv v^2 \ \text{mod } 2$ ensures that $(K,p)$ is actually characteristic.)
 
\begin{prop} $\partial_t \circ \psi_v= \psi_v  \circ \partial_t$.  
\end{prop}
\begin{proof} One computes 
\begin{align*}
\psi_v\partial_t[K,E]&= \sum_{ w \in E} q^{a_w(t)}\otimes \psi_v[K,  E] + \sum_{ w \in E} q^{b_w(t)}\otimes\psi_v[K, E]\\
&=\sum_{p \equiv v^2 (mod \ 2),\ w \in E} q^{a_w(t)}\otimes[K, p, E] + \sum_{p \equiv v^2 (mod \ 2),\ w \in E} q^{b_w(t)}\otimes[K, p, E] \ ,
\end{align*}
where $a_w(t)= (2-t)a_w[K,E]+t a_w[K+2v_0^*, E]$, and $b_w(t)= (2-t)b_w[K,E]+t b_w[K+2v_0^*, E]$. On the other hand,
\begin{align*}
\partial_t \psi_v[K,E]&= \sum_{p \equiv v^2 (mod \ 2)} \partial_t[K, p, E] \\
&=\sum_{p \equiv v^2 (mod \ 2),\ w \in E} q^{a'_w(t)}\otimes[K, p, E] + \sum_{p \equiv v^2 (mod \ 2),\ w \in E} q^{b'_w(t)}\otimes[K, p, E] \ ,
\end{align*}
where $a'_w(t)= (2-t)a_w[K,p,E]+t a_w[K+2v_0^*,p, E]$, and $b'_w(t)= (2-t)b_w[K,p,E]+t b_w[K+2v_0^*,p, E]$. Comparing the exponents the claim boils down to the following identities 
\[
\begin{cases}
a_w[K,E]&= a_w[K,p,E]\\
b_w[K,E]&= a_w[K,p,E]\\
\end{cases} \ \ \ 
\text{ and } \ \  \
\begin{cases}
a_w[K+2v_0^*,E]&= a_w[K+2v_0^*,p,E]\\
b_w[K+2v_0^*,E]&= a_w[K+2v_0^*,p,E]\\
\end{cases}
\]
these follow from the fact that 
\[f_{G-v}[K,E]=f_G[K,p,E]\ \ \ 
\text{ and } \ \  \	f_{G-v}[K+2v_0^*,E]=f_G[K+2v_0^*,p,E]\]
when the vertex $v$ does not belong to the vertex set  $E$.
\end{proof}

Thus, in \eqref{shortexact} we can take $A_t=\psi_v$. Recall that the graph $\Gamma'(v)$ is obtained from $\Gamma$ by adding a $(-1)$-framed vertex $e$. We define $B_t:\mathbb{CF}_t(\Gamma)  \to \mathbb{CF}_t(\Gamma_{+1}(v))$ as the composition of the map $\psi_e:\mathbb{CF}_t(\Gamma)  \to \mathbb{CF}_t(\Gamma'(v))$ with the map $P_t:\mathbb{CF}_t(\Gamma'(v)) \to \mathbb{CF}_t(\Gamma_{+1}(v))$ defined as follows. 

Let $[K, p, 2m-1, E]$ be the generator of $\mathbb{CF}_t(\Gamma'(v))$ associated to the vertex set $E$, and the characteristic vector $(K, p, 2m-1)\in \Char(\Gamma'(v))$ extending $K \in \Char(G-v)$ to $\Gamma'(v)$ so that $K(v)=p$, and $K(e)=2m-1$. Then we define:
\[P_t[K,p,2m-1]=
\begin{cases}
q^{s_m(t)} \otimes [K, p+2m-1,E] \text{ if } e \not\in E\\
 \ \ \ \ \ \ \text{zero otherwise}  
\end{cases} \ ,\]
where 
\[s_m(t)={}^tg_{G_{+1}(v)}[K, p+2m-1, E]-{}^tg_{G}[K, p, 2m-1, E]+m(m-1) \ .\]

\begin{prop}$P_t$ is well-defined.
\end{prop}
\begin{proof}
We must show that $s_m(t)\geq0$. Indeed, $s_m(t)\geq0$ iff both
\begin{equation}\label{A}
g_{G_{+1}(v)}[K, p+2m-1, E]-g_{G}[K, p, 2m-1, E]\geq -\frac{m(m-1)}{2} \ ,
\end{equation} 
and
\begin{equation}\label{B}
g_{G_{+1}(v)}[K+2v_0^*, p+2m-1, E]-g_{G}[K+2v_0^*, p, 2m-1, E]\geq -\frac{m(m-1)}{2} \ . 
\end{equation}
There are two cases. If $v\not\in E$ then for every $I \subseteq E$ we have identities:
\begin{align*}
&f_{G_{+1}(v)}[K, p+2m-1, I]=f_{G}[K, p, 2m-1, I] \\
&f_{G_{+1}(v)}[K+2v_0^*, p+2m-1, I]=f_{G'(v)}[K+2v_0^*, p, 2m-1, I] 
\end{align*}
Hence $s_m(t)=m(m-1)\geq 0$. If on the other hand $v\in I \subseteq E$ we have that  
\begin{align*}
&f_{G_{+1}(v)}[K, p+2m-1, E]-f_{G}[K, p, 2m-1, E] = m  \\ 
&f_{G_{+1}(v)}[K+2v_0^*, p+2m-1, E]-f_{G'(v)}[K+2v_0^*, p, 2m-1, E]=m  \end{align*} 
and we conclude that also in this case $s_m(t)\geq0$.
\end{proof}

\begin{lem} \label{shortexactlemma} The sequence 
\begin{equation}\label{shortexact}
\xymatrix{
0 \ar[r] &\mathbb{CF}_t(\Gamma-v)  \ar[r]^{ \ \ \  A_t  } &  \mathbb{CF}_t(\Gamma)  \ar[r]^{B_t \ \ \ }  & \mathbb{CF}_t(\Gamma_{+1}(v))  \ar[r] & 0} 
\end{equation}
defined as above is short exact.  
\end{lem}
\begin{proof}
One computes
\begin{equation}\label{pluto}
B_t\circ A_t \ [K,E]=\sum_{m= -\infty}^{+\infty} \  \sum_{p = v^2 (mod \ 2)}   q^{m(m-1)} \otimes [K,p+2m-1,E]
\end{equation} 
where $s_m(t)\equiv m(m-1)$ since ${}^tg_{G_{+1}(v)}[K, p+2m-1, E]={}^tg_{G}[K, p, E]$ in the eventuality that $v$ is not in $E$. Moreover the sum on the right hand side of \eqref{pluto} vanishes since the term corresponding to the parameter $(p,m)$ cancels in pair with the one with parameters $(p+4m-2,-m+1)$. (Here plays the fact that we are using $\F$-coefficients.) Thus, $B_t\circ A_t=0$ proving that the sequence in \eqref{shortexact} forms a five-term chain complex. Denote with $H^-_\Delta(G,v)$ its homology.

If we prove that $H^-_\Delta(G,v)=0$ we are done. A direct proof of this vanishing results can be lengthy and somehow confusing, we will argue along another line. First we note that plugging in $q=0$ in \eqref{shortexact} we get a "time independent" chain complex over the base field $\F$
\begin{equation}\label{paperino}
\xymatrix{
0 \ar[r] &\widehat{\mathbb{CF}}_t(\Gamma-v)  \ar[r]^{ \ \ \  \widehat{A}_t  } &  \widehat{\mathbb{CF}}_t(\Gamma)  \ar[r]^{\widehat{B}_t \ \ \ }  & \widehat{\mathbb{CF}}_t(\Gamma_{+1}(v))  \ar[r] & 0} 
\end{equation}  
This can be verified to be a short exact sequence directly, adapting the computations of \cite[Proposition 6.6]{OSS1}. 

Thus the five-term chain complex in \eqref{paperino} is acyclic. If $\widehat{H}_\Delta(G,v)$ denotes its homology, then we have that $H^-_\Delta(G,v)\otimes_\mathcal{R} \F= \widehat{H}_\Delta(G,v)=0$ by the Universal Coefficients Theorem. (A long power series $f=\sum_{\alpha \in A}q^\alpha \in \mathcal{R}$ acts on $x\in \F$ by multiplication by its constant term $f(0)$.) Based on this we can argue for the vanishing of $H^-_\Delta(G,v)$ as follows: suppose by contradiction that $H^-_\Delta(G,v)\not=0$, then a non-trivial element $x \in H^-_\Delta(G,v)$ generates a graded $\mathcal{R}$-module $M \subset H^-_\Delta(G,v)$. Because of the special nature of the ring $\mathcal{R}$ the module $M$ can be decomposed as sums of cyclic modules of the form $\mathcal{R}/q^\alpha$  \cite[Lemma 5.3]{OSS4}, from where the contradiction since $\mathcal{R}/q^\alpha \otimes_\mathcal{R} \F=\F$. 

To see this note that  $\mathcal{R}/q^\alpha$ is the vector space of long power series $f=\sum_{\gamma \in \Omega} \W^\gamma$ with  $\gamma\leq \alpha$. These can be divided into two equivalence classes: those such that $f(0)=0$ and those such that $f(0)=1$. If $f=\sum_{\gamma \in \Omega} \W^\gamma$ is of the first kind then 
\[f\otimes 1 = 1\otimes f(0) \cdot 1= 1\otimes 0=0 \ , \] 
otherwise $f\otimes 1 = 1\otimes f(0) \cdot 1= 1\otimes 1$.   
\end{proof}

\begin{rem}
Note the perfect analogy of the argument presented in the proof of Lemma \ref{shortexactlemma} and the one for the exact triangle in the analytic theory. A similar argument appeared in \cite[Theorem 6.4]{OSS1}.
\end{rem}

\section{First consequences of the exact triangle}
\subsection{Floer simple knots} We briefly explore some immediate consequences of the exact triangle. First recall that given a knot $(Y,K)$ there is a spectral sequence starting with $\widehat{HFK}(Y,K)$ and converging to $\widehat{HF}(Y)$, the Heegaard Floer homology of the ambient three-manifold. This leads to a rank inequality:
\[\text{rk} \widehat{HFK}(Y,K) \geq \text{rk}\widehat{HF}(Y) \ . \]
If the equality holds then we say that $(Y,K)$ is \textit{Floer simple}. If $Y$ is assumed to be an $L$-space this is the same as saying that $tHFK^-(Y,K) = \mathcal{R}^{|H_1(Y; \Z)|}$.

\begin{exa} The plumbing 
\vspace{-0.3cm}
\[\xygraph{ 
!{<0cm,0cm>;<1cm,0cm>:<0cm,1cm>::}
!~-{@{-}@[|(2.5)]}
!{(0,0) }*+{\bullet}="x"
!{(-1.5,0) }*+{\bullet}="a1"
!{(0.5,0) }*+{K}
!{(-1.5,0.5) }*+{-p}
"x"-"a1"
} \ .\]
represents a Floer simple knot in the lens space $L(p,1)$.  
\end{exa}

Floer simple knots played an important role in the work of Baker, Grigsby and Hedden \cite{}. We use the exact triangle to produce some new examples of knots belonging to this class.   

\begin{thm} \label{simpleknots}Suppose that a knot $(Y,K)$ can be represented by means of a negative-definite plumbing tree $\Gamma$ with one unframed vertex $v_0$ and no bad points. Then
\[tHFK^-(Y(G),K)\simeq t\mathbb{HFK}_*(\Gamma)=\mathcal{R}^{|H_1(Y(G);\Z)|} \ .\]
In particular, the knot $(Y,K)$ is Floer simple.
\end{thm}
\begin{proof} First suppose that the strict inequality $\deg(w)<-w^2$ holds at every vertex $w$ of $G=\Gamma\setminus \{v_0\}$. Consider the long exact sequence of Theorem \ref{latticeexact}
\begin{equation*} \xymatrix{
 \ar[r] & t\mathbb{HFK}_p(\Gamma-v) \ar[r] &  t\mathbb{HFK}_p(\Gamma) \ar[r]  & t\mathbb{HFK}_p(\Gamma_{+1}(v))  \ar[r] & t\mathbb{HFK}_{p-1}(\Gamma) \ar[r] &} \ ,
\end{equation*} 
and choose $v$ to be a leaf. Then going by induction as in \cite[Lemma 2.11]{OS20} we can conclude that $t\mathbb{HFK}_p(\Gamma)=0$ for $p>0$, and $t\mathbb{HFK}_0(\Gamma)=\mathcal{R}^{|\det(G)|}$. 

If  $\deg(w)=-w^2$ at some vertices $\{v_{i_1}, \dots, v_{i_k}\}$ of $G$, we can proceed by induction on $k$ choosing $v$ to be one of these vertices. This proves that 
\[t\mathbb{HFK}_*(\Gamma)=\mathcal{R}^{|\det(G)|} = \mathcal{R}^{|H_1(Y(G);\Z)|}\] 
when there are no bad points. On the other hand, in  \cite{Lspaceknots} Ozsv\' ath, Stipsicz, and Szab\' o showed that when there are no bad points  there is a chain homotopy equivalence $\mathbb{CFK}_*^\infty(\Gamma)\simeq CFK^\infty(Y(G),K)$. Since $tHFK^-(Y(G),K)$ and $t\mathbb{HFK}_*(\Gamma)$ are obtained by $t$-modification \cite[Section 4]{OSS4} from $\mathbb{CFK}_*^\infty(\Gamma)$ and $CFK^\infty(Y(G),K)$ respectively, we can conclude that there is an isomorphism  $tHFK^-(Y(G),K)\simeq t\mathbb{HFK}_*(\Gamma)$.
\end{proof}

\begin{exa}The trefoil knot $(S^3, K)$is not Floer simple. Indeed, for the trefoil knot $\text{rk} \widehat{HFK}(S^3,K)=3$, while $\text{rk}\widehat{HF}(S^3)=1$. Note that the trefoil knot can be represented by means of a plumbing tree with one bad vertex. 
\end{exa}
  
\subsection{Reduced lattice groups} Given a $CW$-complex $X$ one can consider the augmentation homomorphism $\epsilon: C_0(X;\F) \to \F$. This gives rise to the reduced homology  $\widetilde{H}_*(X; \F)$ of $X$ \cite{AH}. Similarly given a plumbing diagram $\Gamma$ of an algebraic knot $(Y,K)$ we can consider the augmentation homomorphism
\[\mathbb{CF}_0(G, \s) =\prod_{l\in \R}C_0(M_l; \F) \longrightarrow \prod_{l\geq \gamma(t)} \F=\mathcal{R}_{(\gamma(t))} \ ,\]
where, in the notation of Section \ref{construction}, $\gamma(t)=\min_{x\in \Z^s} \chi_t(x)$. This gives rise to the reduced lattice group $t\mathbb{HFK}_\text{red,*}(\Gamma,\s)=  \prod_{l\in \R}\widetilde{H}_*(M_l; \F)$  
with the property that: 
\[t\mathbb{HFK}_*(\Gamma,\s)=\mathcal{R}_{(\Upsilon_{\Gamma, \s}(t))} \oplus t\mathbb{HFK}_\text{red,*}(\Gamma,\s) \ .\]
In particular, we have that $t\mathbb{HFK}_\text{red,p}(\Gamma,\s)= t\mathbb{HFK}_\text{p}(\Gamma,\s)$ for $p\geq 1$.

\begin{cor}\label{badpoints}If $\Gamma$ represents a type-n knot then $t\mathbb{HFK}_{\text{red},p}(\Gamma)=0$ for $p\geq n$. 
\end{cor} 
\begin{proof}By induction on the number of bad points as in \cite[Theorem 5.1]{OSS1}. The base step here is provided by Theorem \ref{simpleknots}.
\end{proof}  

\section{A map $t\mathbb{HFK}^-_0(\Gamma)\to t \text{HFK}^-(Y(G), K)$}
We start with an handy description of the group $t\mathbb{HFK}^-_0(\Gamma)$.   

\begin{lem}\label{relations}
The group $t\mathbb{HFK}^-_0(\Gamma)$ can be identified with the quotient
\[t\mathbb{HFK}^-_0(\Gamma)= \left. \left( \bigoplus_{k \in \Char(G)} \mathcal{R} \otimes k \right) \right/ \sim \] 
where $\sim$ denotes the equivalence relation generated by the following elementary relations. Let $v$ be a vertex of $G=\Gamma-v_0$, $k\in \Char(G)$ be a characteristic vector, and let $2n=k(v)+v\cdot v$. If $v$ is not connected to $v_0$ in $\Gamma$ then  
\begin{itemize}
\item if $n\geq 0$ then $q^{2n+m}\otimes (k+2v^*)\sim q^{m}\otimes k$,
\item while if $n\leq 0$ then $q^{m}\otimes (k+2v^*)\sim q^{m-2n}\otimes k$.
\end{itemize} 
If $v$ is connected to $v_0$ in $\Gamma$ then
\begin{itemize}
\item if $n\geq 0$ then $q^{m+2n+t}\otimes (k+2v^*)\sim q^{m}\otimes k$,
\item while if $n\leq -1$ then $q^{m}\otimes (k+2v^*)\sim q^{m-2n-t}\otimes k$.
\end{itemize} 
\end{lem}
\begin{proof}
 This is just a direct computation of the differential $\partial_t: \CF_1(G)\to \CF_0(G)$.
\end{proof}

Based on this explicit description of the lattice group $t\mathbb{HFK}^-_0(\Gamma)$ we  build a module homomorphism $t\mathbb{HFK}^-_0(\Gamma)\to t \text{HFK}^-(Y(G), K)$. 

First we observe that the algebraic knot associated to a plumbing diagram $\Gamma$ comes with an associated doubly pointed Heegaard diagram. This is obtained as follows. Fix a planar drawing $\Gamma \hookrightarrow \R^2$ of the graph $\Gamma$ and build the associated surgery diagram $\mathcal{D}_\Gamma$ suggested by Figure \ref{trefoil}. Forgetting about the framings and the under-over conventions this gives a connected, $4$-valent planar graph. Thickening $D_\Gamma \subset \R^2\times 0 \subset \R^3$ we get a genus $g=\#(\text{vertices})+ \# (\text{edges})$ solid handlebody $V \subset S^3$ providing, together with its complement $V^c=\overline{S^3-V}$, a Heegaard splitting of $S^3$. We choose as $\alpha$-curves of the corresponding Heegaard diagram (supported on $\Sigma= \partial V$) the boundary of the compact complementary regions of $\mathcal{D}_\Gamma \subset \R^2$ (this provides a set of compressing circles for $U_\alpha=V^c$). As compressing circles of $U_\beta=V$ ($\beta$-curves) we take a meridian for each link component of $\mathcal{D}_\Gamma$, and one further curve in correspondence of each edge of $\Gamma$ as suggested by Figure \ref{beta}. By taking base points on the two sides of the $\beta$-curve corresponding to the meridian of the unlabled vertex $v_0\in \Gamma$ we get a doubly pointed Heegaard diagram $\mathcal{H}_{\alpha\beta}=(\Sigma, \alpha, \beta, z, w)$ representing the unknot $U \subset S^3$. 

\begin{figure}[t]
\begin{center}
\includegraphics[scale=0.7]{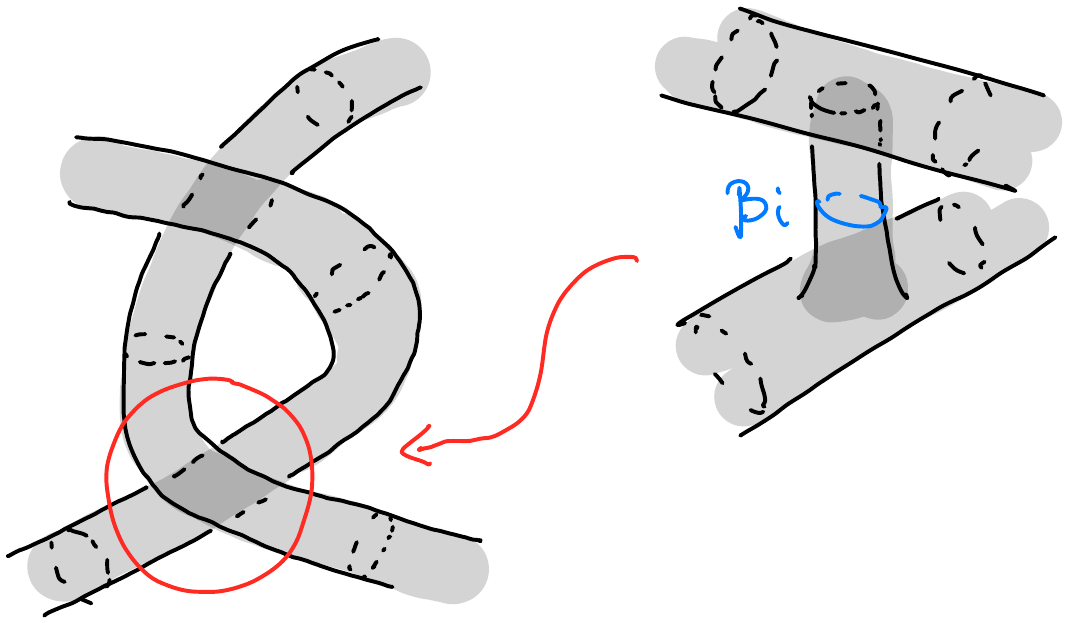}
\caption{\label{beta}}
\end{center}
\end{figure}

Notice that in the choice of the $\beta$-curves above we can substitute the meridians of the framed components with the framings. By replacing the other curves with small Hamiltonian translates of the remaining $\beta$-curves we get another $g$-tuple  of curves $\gamma$. This gives rise to other two doubly-pointed Heegaard diagrams: $\mathcal{H}_{\alpha\gamma}=(\Sigma, \alpha, \gamma, z, w)$ representing $K \subset Y(G)$, and $\mathcal{H}_{\beta\gamma}=(\Sigma, \beta, \gamma, z, w)$ representing the unknot in $\#^{g-\ell} S^1 \times S^2$.

Let $\Theta_{\beta\gamma} \in \Tt_\beta \cap \Tt_\gamma$ denote the intersection point representing the top-dimensional generator of $H_*(CF_t(\Tt_\beta, \Tt_\gamma))= \Lambda^*H_1(T^{g-\ell}) \otimes \mathcal{R}$. Given intersection points $\x \in \Tt_\alpha \cap \Tt_\beta$ and $\y \in \Tt_\alpha \cap \Tt_\gamma$ Ozsv\' ath and Szab\' o \cite[Proposition 8.4]{OS2}  built a map $\s_z : \pi_2(\x, \Theta_{\beta \gamma} , \y) \to \Spinc(X(G))$. Given a $\Spinc$ structure $\mathfrak{t}$ of the plumbing of spheres $X(G)$ we define a map 
$f_{\Gamma, \mathfrak{t}}: CF_t(\Tt_\alpha, \Tt_\beta)  \to CF_t(\Tt_\alpha, \Tt_\gamma)$  
setting
\[f_{\Gamma, \mathfrak{t}}(\x)=\sum_{\y \in \Tt_\alpha \cap \Tt_\gamma} \sum_{ \substack{ \Delta \in \pi_2(\x ,\Theta_{\beta\gamma}, \y) \\   \s_z(\Delta)=\mathfrak{t}, \ \mu( \Delta ) =0  }}\#\mathcal{M}(\Delta) \ \W^{tn_z(\Delta)+ (2-t) n_w(\Delta)} \cdot \y  \ .\]
Of course, this is just a $\Spinc$ refinement of the holomorphic triangle count of Section \ref{sectionanalytictriangle}. Inspecting the ends of moduli spaces of holomorphic triangles with Maslov index $\mu=1$ one  concludes that $f_{\Gamma, \mathfrak{t}}$ is a chain map. Thus given a characteristic vector $k \in \Char_\s(G)\subset \Spinc(X(G))$ we get a map 
$F_{\Gamma, k}: \mathcal{R}=H_*(CF_t(\Tt_\alpha, \Tt_\beta)) \to \Kt(K, Y(G), \s)$, $F_{\Gamma, k}= (f_{\Gamma, k})_*$.

\begin{exa}[Lens space cobordism]\label{lensspaces}
Let $D_{-p}$  denote the total space of the disk bundle $D_{-p}\to S^2$ with Euler number $e=-p$, and let $\Delta \subset D_p$ be a fibre disk. Then $\partial(D_{-p}, \Delta)= (L(p,1), K)$ is a Floer simple knot with plumbing 

\vspace{-0.3cm}
\[\xygraph{ 
!{<0cm,0cm>;<1cm,0cm>:<0cm,1cm>::}
!~-{@{-}@[|(2.5)]}
!{(0,0) }*+{\bullet}="x"
!{(-1.5,0) }*+{\bullet}="a1"
!{(0.5,0) }*+{K}
!{(-1.5,0.5) }*+{-p}
"x"-"a1"
} \ .\]

A doubly pointed Heegaard triple representing  $(D_{-p}, \Delta)$ is given by $(T^2, \alpha, \beta, \gamma, z, w)$, where $\alpha$ and $\gamma$ represent respectively a longitude and a meridian of the two torus $T^2$, and $\beta$ a type $(1,p)$ curve. The two base points $z$ and $w$ are chosen to lie on the two opposite sides of the $\beta$-curve.

Let $v$ represents the generator of $H_2(D_{-p}, \Z)$. If $k$ is a characteristic vector such that $2n=k(v)+v\cdot v$ then
\begin{itemize}
\item  $q^{2n+t}F_{\Gamma, \ k+2v^*}=F_{\Gamma,k}$ if $n\geq 0$, and 
\item  $F_{\Gamma, \  k+2v^*}= q^{-2n-t}F_{\Gamma,k}$ if $n\leq -1$.
\end{itemize}  
This is a direct computation along the line of \cite[Section 4.1]{correctiontermslens}.
\end{exa}

We define $\varphi_{\Gamma}: \Char_\s(G) \to \Kt(Y(G),K, \s)$ setting $\varphi_{\Gamma}(k)=F_{\Gamma, k}(1)$. 

\begin{lem} $\varphi_{\Gamma}$  descends to an $\mathcal{R}$-module homomorphism 
\[\Phi_{\Gamma}:t\mathbb{HFK}^-_0(\Gamma, \s)\to t\text{HFK}^-(Y(G),K, \s) \ .\]
\end{lem}
\begin{proof}A vertex $v$ of $G$ corresponds to a $(-p)$-framed sphere embedded in $X(G)$. This has a tubular neighbourhood $D_{-p}$ diffeomorphic to a disk bundle with base the two-sphere $S^2$, and Euler number $e=-p$. Let $\Delta \subset D_p$ be a fibre disk. 

There are two cases: the case when $v$ is linked to the unframed vertex $v_0$ and the case when it is not. In the case when $v$ is not linked to $v_0$ in $\Gamma$ one concludes that the map $F_{\Gamma, k}:\mathcal{R} \to tHFK^-(K, Y(G), \s)$ factors through the cobordism map associated to $D_{-p}: S^3 \to L(p,1)$. This implies the first set of relations of Lemma \ref{relations}. If $v$ is linked to $v_0$ on the other hand, the map $F_{\Gamma, k}:\mathcal{R} \to tHFK^-(K, Y(G), \s)$ factors through the cobordism map associated to the pair $(D_{-p}, \Delta)$ we computed in Example~\ref{lensspaces}. In this case we get the second set of relations.    
\end{proof}

\begin{prop}If $\Gamma$ has at most one bad point then there is an isomorphism of $\mathcal{R}$-modules $tHFK^-(Y(G), K) \simeq t\mathbb{HFK}^-_*(\Gamma)$.
\end{prop}
\begin{proof} 
Putting the pieces together we get a commutative diagram with exact rows: 
\[\xymatrix{
& tHFK^-(Y(G-v), K_0) \ar[r]  &tHFK^-(Y(G), K) \ar[r]  &  tHFK^-(Y(G_{+1}(v)), K') \\ 
&t\mathbb{HFK}^-_0(\Gamma-v) \ar[r] \ar[u]^{\Phi_{\Gamma-v}}&t\mathbb{HFK}^-_0(\Gamma) \ar[r] \ar[u]^{\Phi_{\Gamma}} & t\mathbb{HFK}^-_0(\Gamma_{+1}(v))  \ar[u]^{\Phi_{\Gamma_{+1}(v)}} \\
}\]
as in \cite[Lemma 2.10]{OS20}. Note that the bottom row fits in a short exact sequence 
\begin{equation*} \xymatrix{
 0\ar[r] & t\mathbb{HFK}_0(\Gamma-v) \ar[r] &  t\mathbb{HFK}_0(\Gamma) \ar[r]  & t\mathbb{HFK}_0(\Gamma_{+1}(v))  \ar[r]  &0} \ ,
\end{equation*}  
provided that $t\mathbb{HFK}_1(\Gamma_{+1}(v))=0$. Indeed, one can show that similarly the top row fits in a short exact sequence 
\[0 \to tHFK^-(Y(G-v), K_0) \to tHFK^-(Y(G), K) \to  tHFK^-(Y(G_{+1}(v)), K')\to 0 \ ,\] 
provided that $G_{+1}(v)$ is negative definite and $G-v$ has no bad points. To see this we acknowledge that $tHFK^\infty(Y,K)$, the homology of the localisation $CF_t(Y,K)\otimes \mathcal{R}_*$, agrees with $HF^\infty(Y)\otimes \mathcal{R}_*\simeq \mathcal{R}_*$ and that the map 
\[
\xymatrix{
HF^\infty(Y(G_{+1}(v))) \otimes \mathcal{R}_*  \ar[r]  &  HF^\infty(Y(G-v)) \otimes \mathcal{R}_* \\ 
tHFK^\infty(Y(G_{+1}(v)), K') \ar[r] \ar[u]^{\cong}& tHFK^\infty(Y(G-v), K_0) \ar[u]^{\cong} \\
}
\]
is trivial being the underlying cobordism indefinite. 

With the above said one can run the very same argument of \cite[Theorem 2.1]{OS20} to show that there is an identification  $tHFK^-(Y(G), K) \simeq t\mathbb{HFK}^-_0(\Gamma)$. On the other hand, according to Lemma \ref{badpoints}, for a diagram with at most one bad point we have that $t\mathbb{HFK}^-_p(\Gamma)=0$ for $p>0$. Hence $t\mathbb{HFK}^-_*(\Gamma)=t\mathbb{HFK}^-_0(\Gamma)$, proving the claim.
\end{proof}

\subsection{$\Z/2\Z$-grading and extension to the case of two bad points} We now want to extend our main result to the case of two bad points. To this end we introduce the relative $\Z/2\Z$-grading on \textit{some} of the $t$-modified knot homologies.

First, note that $tHFK^-(Y,K)$ has a natural real-valued relative $\gr_t$-grading 
\[\gr_t(x,y)=\gr_t(x)-\gr_t(y) \ ,\]
characterised by the property that for every pair of intersection points $\x$ and $\y$
\[\gr_t(\x,\y)= \mu(\phi)+t\cdot n_z(\phi)+(2-t) \cdot n_w(\phi) \ , \]
where $\phi\in \pi_2(\x,\y)$ denotes a Whitney disk connecting $\x$ to $\y$. Analogously, in the combinatorial theory $t\mathbb{HFK}_*$ we have a relative $\gr_t$-grading
\begin{align*}
\gr_t(\Box,\Box')&= \dimm(\Box)- \dimm(\Box') \\ 
&+t \cdot (w_k(\Box)-w_k(\Box'))+ (2-t) \cdot (w_{k+2v_0^*}(\Box)-w_{k+2v_0^*}(\Box')) 
\end{align*}

Secondly, recall \cite[Section 3]{OSS4} that if $0\leq t=\frac{m}{n}\leq 2$ is a rational number then in the definition of  $tHFK^-(Y,K)$ the ring $\mathcal{R}$ can be replaced with the ring of polynomials with fractional exponents $\F[q^{1/n}] \subset \mathcal{R}_*$. In this case the relative $\text{gr}_t$-grading is rational valued. More specifically, it takes values in $\Z[\frac{1}{n}] \subset \Q$, the ring of fractions associated to the multiplicative set $S=\{n^k: k\geq 0\}\subset \Z$. Note that the ideals of $\Z[\frac{1}{n}]$ are in one-to-one correspondence with the ideals of $\Z$ that do not meet $S$. In particular if $n$ is \emph{odd} then $\Z[\frac{1}{n}]$ contains a (unique) maximal ideal $\mathfrak{m}$ with quotient field $\Z[\frac{1}{n}]/\mathfrak{m} \simeq \Z/2\Z$. Thus, if we choose $0\leq t\leq 2$ to be of the form $t=m/(2b+1)$ then there is a well defined $(\text{mod }2)$ reduction of the relative $\gr_t$-grading. The equivalence relation: 
\[\x \sim \y \ \  \Longleftrightarrow \ \ \gr_t(\x,\y)=0 \ (\text{mod }2)\]
partitions the generators of $tHFK^-(Y,K)$ into two equivalence classes, and determines a splitting of the chain 
group: $CF_t(Y,K)= CF^\text{even}_t(Y,K) \oplus CF^\text{odd}_t(Y,K)$. Since the differential $\partial_t$ flips the two summands, we have a decomposition of the groups 
\[tHFK^-(Y,K)= tHFK^-_\text{even}(Y,K) \oplus tHFK^-_\text{even}(Y,K) \ . \]
Similarly one has a splitting in the combinatorial theory
\[t\mathbb{HFK}_*(\Gamma)= t\mathbb{HFK}_\text{even}(\Gamma) \oplus t\mathbb{HFK}_\text{odd}(\Gamma) \ . \]

Thirdly, we observe that if we choose $t=m/(2b+1)$ with $m=2a$ \textit{even} then the situation greatly simplifies. In the analytic theory the relative $\Z/2\Z$-grading  $\gr_t(\x,\y)= \mu(\phi) \ (\text{mod }2)$ collapses on the "classical" Maslov grading, while in the combinatorial theory there are identifications:
\[t\mathbb{HFK}_\text{even}(\Gamma)= \bigoplus_{p \text{ even}}  t\mathbb{HFK}_{p}(\Gamma) \ ,   \text{ and } \  
t\mathbb{HFK}_\text{odd}(\Gamma)= \bigoplus_{p \text{ odd}} t\mathbb{HFK}_{p}(\Gamma) \ ,\] 
since $\gr_t(\Box, \Box')=\dimm(\Box)- \dimm(\Box') \ (\text{mod }2)$. Note that in the case when $\Gamma$ has at most one bad point we have a vanishing result for the odd homologies: $tHFK^-_\text{odd}(Y(G), K)\simeq t\mathbb{HFK}^-_\text{odd}(\Gamma)=0$.

\begin{prop}\label{close}Suppose that $0\leq t \leq 2$ is a rational number of the form $t=2a/(2b+1)$. If $\Gamma$ has at most two bad points then there is an isomorphism of $\mathcal{R}$-modules $tHFK^-_\text{even}(Y(G), K) \simeq t\mathbb{HFK}^-_0(\Gamma)$.
\end{prop}
\begin{proof} 
If we choose $v$ to be one of the two bad points we get a commutative diagram with exact rows: 
\[\xymatrix{
tHFK^-_\text{even}(Y(G-v), K_0) \ar[r]  &tHFK^-_\text{even}(Y(G), K) \ar[r]  &  tHFK^-_\text{even}(Y(G_{+1}(v)), K') \\ t\mathbb{HFK}^-_0(\Gamma-v) \ar[r] \ar[u]^{\Phi_{\Gamma-v}}&t\mathbb{HFK}^-_0(\Gamma) \ar[r] \ar[u]^{\Phi_{\Gamma}} & t\mathbb{HFK}^-_0(\Gamma_{+1}(v))  \ar[u]^{\Phi_{\Gamma_{+1}(v)}} \\
}\]
Again, we have that the bottom row fits in a short exact sequence 
\begin{equation*} \xymatrix{
 0\ar[r] & t\mathbb{HFK}_0(\Gamma-v) \ar[r] &  t\mathbb{HFK}_0(\Gamma) \ar[r]  & t\mathbb{HFK}_0(\Gamma_{+1}(v))  \ar[r]  &0} \ .
\end{equation*}  
Furthermore the last map on the top row is surjective since the triangle map 
\[tHFK^-(Y(G_{+1}(v)), K')\to tHFK^-(Y(G-v), K_0)\] 
flips the relative Maslov grading and $tHFK^-_\text{even}(Y(G-v), K_0)=0$. Like in \cite[Theorem 2.2]{OS20} this information is sufficient to conclude that the right-most map should be an isomorphism provided that the other two are. 
\end{proof}

\begin{proof}[Proof of Theorem \ref{maingoal}]
First suppose that $0\leq t \leq 2$ is a rational number of the form $t=2a/2b+1$. According to the computations of the degree shifts of the triangle counting maps performed by Zemke in \cite{zemkegradings} the isomorphism $tHFK^-_\text{even}(Y(G),K)\simeq  t\mathbb{HFK}_0(\Gamma)$ of Proposition \ref{close} preserves the $\gr_t$-grading. Since non-torsion elements are concentrated in even gradings we have that $\Upsilon_{K,\s}(t)=\Upsilon_{\Gamma, \s}(t)$ for all values of the parameter $t$ in 
\[C=\left\{\frac{2a}{2b+1} \ : \ a, b\in \Z_{\geq 0} \right\}  \cap [0,2]\ . \]
On the other hand $C$ is dense in $[0,2]$, the upsilon function is continuous \cite[Proposition 1.4]{OSS4},  and two continuous functions  that agree on a dense set agree everywhere.  
\end{proof}

\bibliography{Bibliotesi}
\bibliographystyle{siam}

\end{document}